\documentclass[10pt,letterpaper,reqno]{amsart}
\usepackage{amsfonts, amsmath, amssymb, amscd, amsthm, graphicx,enumerate}
\usepackage{color}
\usepackage{hyperref}

\makeatletter
\def\blfootnote{\xdef\@thefnmark{}\@footnotetext}
\makeatother

\hypersetup{colorlinks=true,linkcolor=red,citecolor=blue,urlcolor=blue}

\newtheorem{thm}{Theorem}[section]
\newtheorem{mainthm}{Theorem}

\newtheorem{cor}[thm]{Corollary}
\newtheorem{lem}[thm]{Lemma}
\newtheorem{prop}[thm]{Proposition}

\theoremstyle{definition}
\newtheorem{defn}[thm]{Definition}
\theoremstyle{remark}
\newtheorem{rem}[thm]{Remark}

\newtheorem{conv}[thm]{Convention}

\numberwithin{equation}{section}

\renewcommand{\phi}{\varphi}
\renewcommand{\epsilon}{\varepsilon}
\newcommand{\N}{\mathbb N}
\newcommand{\R}{\mathbb R}
\newcommand{\Z}{\mathbb Z}
\newcommand{\SC}{\operatorname{SC}}
\newcommand{\GC}{\operatorname{GC}}
\newcommand{\supp}{\operatorname{supp}}
\newcommand{\Hull}{\operatorname{Hull}}
\newcommand{\diam}{\operatorname{diam}}

\newcommand{\Cay}{\operatorname{Cay}}

\newcommand{\calH}{\mathcal H}
\newcommand{\calT}{\mathcal T}
\newcommand{\calM}{\mathcal M}

\newcommand{\Sat}{\operatorname{Sat}}
\newcommand{\dd}{\partial_2}

\begin{document}
\raggedbottom

\title[Uniquely Ergodic Branching Subset Currents]{Uniquely Ergodic Branching Subset Currents on Free and Surface Groups}

\author[I. Kapovich]{Ilya Kapovich}

\address{Department of Mathematics and Statistics, Hunter College of CUNY\newline
  \indent 695 Park Ave, New York, NY 10065
  \newline \indent  {\url{http://math.hunter.cuny.edu/ilyakapo/}}}
\email{ik535@hunter.cuny.edu}


\date{}
\hypersetup{
  pdftitle={Uniquely Ergodic Branching Subset Currents on Free and Surface Groups},
  pdfauthor={Ilya Kapovich},
  pdfsubject={Branching subset currents: unique ergodicity, Choquet universality, and complexity},
  pdfkeywords={subset currents, free groups, surface groups, unique ergodicity, Choquet simplex, Toeplitz subshifts, rooted trees, complexity}
}

\begin{abstract}
\emph{Subset currents} on free groups, on surface groups and, more generally, on word-hyperbolic groups, were introduced by Kapovich and Nagnibeda, and further studied by Sasaki. Subset currents extend ordinary geodesic currents and provide measure-theoretic generalizations of conjugacy classes of infinite quasiconvex subgroups. Their supports can be viewed as ``branching laminations".

For free groups and surface groups, we construct non-rational subset currents with genuinely branching supports that carry no other subset currents except scalar multiples.  For every nonabelian finite rank free group and every closed hyperbolic surface group, we show that such a current can be chosen so that its support is uncountable, is disjoint from the locus of two-point boundary subsets, and consists entirely of Cantor subsets of the boundary.  More generally, we prove that every nonempty metrizable Choquet simplex occurs as a compact convex base for the cone of currents carried by a minimal genuinely branching subset lamination.  The same universality phenomenon extends, apart from the usual rational-current terminology, to every non-elementary torsion-free hyperbolic group.

We also study finite-pattern complexity.  We prove that if $R_L(n)$ counts the allowed \emph{round patterns} of radius $n$ in a minimal subset lamination $L$, then $R_L(n)=O(n)$ forces the carried-current cone to be finite-dimensional, while infinite-dimensionality forces $R_L(n)/n\to\infty$; this universal threshold is sharp.  We show that the coarse growth type of $R_L$ is independent of the chosen free basis and, more generally, of the chosen finite marked graph.  We also show that a refined switch system recovers the exact dimension of the carried-current cone.
\end{abstract}

\maketitle
\begingroup
\setlength{\parskip}{0pt}
\setcounter{tocdepth}{1}
\tableofcontents
\endgroup

\section{Introduction}
\enlargethispage{4pt}

Geodesic currents, introduced by Bonahon for hyperbolic surfaces~\cite{Bon86,Bon88}, give a measure-theoretic enlargement of the space of closed geodesics.  Bonahon's intersection form extends geometric intersection number, while Liouville currents relate the theory to Teichm\"uller space and Thurston's compactification~\cite{Bon88}.

For a non-elementary word-hyperbolic group $G$, the cone $\GC(G)$ of \emph{geodesic currents} consists of locally finite positive Borel measures on $\partial^2G=\{(p,q)\in\partial G\times\partial G:p\ne q\}$ that are $G$-invariant and flip-invariant.  The standard counting current $\eta_g$ associated to an infinite-order element $g\in G$ generalizes its conjugacy class; nonnegative multiples of counting currents are called \emph{rational}, and are dense in $\GC(G)$~\cite{Bon91}.

For free groups, the theory began with Martin's thesis~\cite{Martin95} and was developed further in~\cite{Kap06,KapLus09,KapLus10}.  Geodesic currents complement Culler--Vogtmann Outer space: the Kapovich--Lustig intersection form pairs very small $F_N$-trees with currents and extends translation length on counting currents~\cite{KapLus09,KapLus10}.  Subset currents replace boundary pairs by arbitrary closed boundary subsets and generalize conjugacy classes of infinite quasiconvex subgroups.

Let $G$ be a non-elementary word-hyperbolic group and let
\[
 \calH(\partial G)=\{S\subseteq \partial G: S \text{ is closed and } \#S\ge 2\}
\]
carry the Vietoris (equivalently, Hausdorff) topology; see Subsection~\ref{subsec:subset-laminations}.  A \emph{subset current} is a $G$-invariant locally finite positive Borel measure on $\calH(\partial G)$; write $\SC(G)$ for their cone (Definition~\ref{def:subset-current}).  Subset currents were introduced by Kapovich--Nagnibeda~\cite{KN13}; Sasaki~\cite{Sas22} developed the theory for closed hyperbolic surface groups.  After passing from ordered to unordered boundary pairs, the geodesic-current cone $\GC(G)$ embeds as the currents supported on
\[
 \dd G:=\{S\in\calH(\partial G):\#S=2\}.
\]
Following~\cite[Section~10.7]{KN13}, we call a closed $G$-invariant subset of $\calH(\partial G)$ a \emph{subset algebraic lamination}, or simply a \emph{subset lamination}.  A subset algebraic lamination is \emph{geodesic} if contained in $\dd G$, \emph{branching} if not contained in $\dd G$, and \emph{genuinely branching} if disjoint from $\dd G$; see Definition~\ref{def:subset-lamination} for these notions and minimality, and Subsection~\ref{subsec:subset-laminations} for convex-hull realizations.  In our constructions every leaf is a Cantor subset of $\partial G$.  We write $L_{\mathrm{full}}(G):=\calH(\partial G)$ for the \emph{full support subset lamination} on $G$.

When defined for an infinite finitely generated subgroup $H\le G$, the counting current $\eta_H$ is the locally finite atomic measure on the $G$-orbit of its limit set $\Lambda(H)$, with the standard multiplicity; see Subsection~\ref{subsec:subset-currents}.  Subset currents of the form $c\eta_H$, where $c\ge 0$ and $H$ as above, are called \emph{rational}. Rational currents are dense in $\SC(G)$ in the free-group case~\cite[Theorem~5.8]{KN13} (see also~\cite[Theorem~5.1]{Kap17}) and in the closed-surface case~\cite[Theorem~8.21]{Sas22}.

Beyond geodesic and rational currents, relatively few branching examples have been studied.  The probabilistically motivated \emph{absolute uniform subset current} of~\cite[Section~9]{KN13} has full support $L_{\mathrm{full}}(F_N)=\calH(\partial F_N)$, which carries \emph{every} subset current.  Its round complexity, defined in Definition~\ref{def:round-switch-complexity}, grows doubly exponentially; see Remark~\ref{rem:full-support-round-complexity}.

Here we construct non-rational uniquely ergodic examples with genuinely branching supports, and more generally study the possible carried-current cones and their dimensions.

A nonzero subset current $\mu$ is called \emph{support-uniquely ergodic} when every subset current supported on $\supp(\mu)$ is a nonnegative scalar multiple of $\mu$; see Definition~\ref{def:SUE} for the formal definition.  We reserve ordinary ergodicity for extremality of an invariant measure or of a carried-current ray.

Our first main result is:

\begin{mainthm}\label{mainthm:main}
Let $G$ be one of the following:
\begin{enumerate}[(1)]
\item $G=F_N$, where $N\ge2$;
\item $G=\pi_1(\Sigma)$, where $\Sigma$ is a closed hyperbolic surface.
\end{enumerate}
Then there exists a nonzero subset current $\mu\in\SC(G)$ with the following properties:
\begin{enumerate}[(a)]
\item $\mu$ is support-uniquely ergodic:
\[
 \{\nu\in\SC(G):\supp(\nu)\subseteq\supp(\mu)\}
 =\R_{\ge0}\mu.
\]
\item $\supp(\mu)\cap\dd G=\varnothing$.  In fact every $S\in\supp(\mu)$ is a Cantor subset of $\partial G$.
\item $\supp(\mu)$ is uncountable.  Consequently $\mu$ is not rational.
\end{enumerate}
\end{mainthm}

Theorem~\ref{mainthm:main} is proved in Subsection~\ref{subsec:proof-main}.

Theorem~\ref{mainthm:main} is a special case of a more flexible universality statement.  For a subset lamination $L\subseteq\calH(\partial G)$, write
\[
 \SC_G(L):=\{\mu\in\SC(G):\supp(\mu)\subseteq L\}
\]
for its carried-current cone; compact convex bases and projectively ergodic currents are defined in Definition~\ref{def:carried-current-cone}.

\begin{mainthm}[Universality]\label{mainthm:universality}
Let $G$ be one of the following:
\begin{enumerate}[(1)]
\item $G=F_N$, where $N\ge2$;
\item $G=\pi_1(\Sigma)$, where $\Sigma$ is a closed hyperbolic surface.
\end{enumerate}
Let $\mathcal K$ be a nonempty metrizable Choquet simplex, in the sense of Definition~\ref{def:choquet-simplex}.  Then there exists an uncountable minimal genuinely branching subset lamination
\[
 L\subseteq\calH(\partial G)
\]
with the following properties:
\begin{enumerate}[(a)]
\item Every $S\in L$ is a Cantor subset of $\partial G$.  In particular $L\cap\dd G=\varnothing$.
\item The cone $\SC_G(L)$ has a compact convex base $\mathfrak B(L)$ affinely homeomorphic to $\mathcal K$.
\item Every nonzero $\mu\in\SC_G(L)$ satisfies $\supp(\mu)=L$.  Consequently every such $\mu$ is non-rational.
\end{enumerate}
Under the affine identification in~(b), the extreme points of $\mathcal K$ correspond exactly to the ergodic rays of $\SC_G(L)$.
\end{mainthm}

Theorem~\ref{mainthm:universality} is proved in Section~\ref{sec:universality}.

The construction begins in $F_4$.  Put $B=F(u,v)$ and let $\Omega\subseteq\{0,1\}^{B}$ be an infinite minimal uniquely ergodic subshift for the right shift of Convention~\ref{conv:tree-coding}.  We encode $x\in\Omega$ as an exactly $4$-regular subtree $T_x$ of the Cayley tree of
\[
 F=F(a_0,a_1,b_0,b_1).
\]
The four incident edges at the vertex representing $w\in B$ remember the values of $x$ at $w$, $wu^{-1}$, and $wv^{-1}$.  The key identity is
\[
 \kappa_x(w)^{-1}T_x=T_{x\cdot w},
\]
where $\kappa_x(w)$ is the vertex corresponding to $w$ and $x\cdot w$ is the shift by $w$.  Thus $\Omega\rtimes B$ is the root-change groupoid of $\{T_x:x\in\Omega\}$; see Proposition~\ref{prop:measure-correspondence}.  The rooted-tree model then turns the unique invariant probability on $\Omega$ into a support-uniquely-ergodic subset current.

An elementary choice pulls back an aperiodic minimal uniquely ergodic $\Z$-subshift, for example the Fibonacci subshift, along $F(u,v)\twoheadrightarrow\Z$; here the $F_2$-action factors through $\Z$.  For genuinely nonabelian dynamics, Cecchi Bernales--Cortez--G\'omez~\cite[Proposition~3.8 and Theorem~1.2]{CCG24} give a free regular uniquely ergodic binary Toeplitz $F_2$-subshift, which cannot factor through a proper quotient.

For $F_N$ with $N=2,3$ and for surface groups, we induce through a malnormal quasiconvex $H\cong F_4$; see Subsection~\ref{subsec:prelim-induction}.  Such a subgroup exists by Kielak--Linton~\cite[Theorem~1.11]{KL24}; compare~\cite{Kap99}, and Sasaki's map~\cite[Theorem~4.3]{Sas22} gives the induced current.  Malnormality makes distinct translated limit sets disjoint and preserves projective uniqueness.

For Theorem~\ref{mainthm:universality}, Downarowicz~\cite[Theorem~1]{Dow91} realizes every nonempty metrizable Choquet simplex as the invariant-measure simplex of a minimal binary Toeplitz flow.  The tree coding identifies its invariant-measure cone with the carried-current cone, and malnormal induction preserves that cone.  Hence the simplex survives unchanged in every group in the theorem.

\begin{rem}\label{rem:broader}
The same induction proof gives support-unique ergodicity and genuine branching for every non-elementary torsion-free hyperbolic group.  In that generality non-rationality should be phrased relative to the available counting currents; for free and closed-surface groups Theorem~\ref{mainthm:main} has the stated form.
\end{rem}

\begin{rem}\label{rem:delta-boundary}
The Dirac mass $\delta_{\partial G}$ is a trivial one-point-support example, but for $F_N$ or a closed surface group it is the counting current of the finite-index subgroup $G$ and hence rational.  Our supports are instead uncountable, genuinely branching, and consist of Cantor subsets of the boundary.
\end{rem}

In the second part of the paper, we investigate how much information about the carried-current cone can be detected from finite-pattern complexity.  Raw finite-subtree complexity (Definition~\ref{def:subtree-complexity}) is too coarse. Round complexity, based on the round trees introduced in~\cite{Kap17}, is more effective.

For the next two statements, fix $F_N=F(A)$, $N\ge2$.  For a minimal subset lamination $L\subseteq\calH(\partial F_N)$ carrying a nonzero current, let $d(L)$ be the dimension of its carried-current span and $R_L(n)$ its round complexity (Definitions~\ref{def:carried-current-cone} and~\ref{def:round-switch-complexity}); for $n\ge2$ put
\[
 E_L(n):=R_L(n)-R_L(n-1),
\]
the round extension excess of Definition~\ref{def:extension-excess}.

\begin{mainthm}[Round complexity and carried-current dimension]\label{mainthm:complexity}
Let $F_N=F(A)$, where $N\ge2$, and let $L\subseteq\calH(\partial F_N)$ be a minimal subset lamination carrying a nonzero subset current.  Then
\[
 \boxed{
 d(L)\le 1+\liminf_{n\to\infty}E_L(n)
 \le 1+\liminf_{n\to\infty}\frac{R_L(n)}n.
 }
\]
Moreover:
\begin{enumerate}[(1)]
\item If $d(L)=d<\infty$, then
\[
 E_L(n)\ge d-1
\]
for all sufficiently large $n$, and hence
\[
 R_L(n)\ge (d-1)n-O(1).
\]
In this case $L$ carries exactly $d$ ergodic rays.
\item If $d(L)=\infty$, then
\[
 E_L(n)\longrightarrow\infty
 \qquad\text{and}\qquad
 \frac{R_L(n)}n\longrightarrow\infty.
\]
\end{enumerate}
In particular,
\[
 R_L(n)=O(n)\quad\Longrightarrow\quad d(L)<\infty.
\]
\end{mainthm}

Theorem~\ref{mainthm:complexity} is proved in Section~\ref{sec:switch}.

Section~\ref{sec:switch} also gives an exact inverse-limit refinement: the stable switch sets are the finite-level images of the normalized carried-current base, and their dimensions recover $d(L)$.

The linear conclusion is universally sharp.  In the next statement $\mathcal M_A(L)$ is the normalized carried-current base of Definition~\ref{def:normalized-current-base}.

\begin{mainthm}[Sharpness of the linear threshold]\label{mainthm:sharp-threshold}
Let $\mathcal K$ be a nonempty metrizable Choquet simplex and let $(a_n)_{n\ge1}$ be a sequence of positive real numbers such that
\[
 \frac{n}{a_n}\longrightarrow0.
\]
Then, for the standard basis $A=\{a_0,a_1,b_0,b_1\}$ of $F_4$, there exists an uncountable minimal genuinely branching subset lamination
\[
 L\subseteq\calH(\partial F_4)
\]
such that
\[
 \mathcal M_A(L)\cong_{\rm aff}\mathcal K
 \qquad\text{and}\qquad
 R_L(n)=o(a_n).
\]
In particular, if $\mathcal K$ has infinite-dimensional affine hull, then $d(L)=\infty$ while $R_L$ lies below the prescribed superlinear scale $(a_n)$.  Hence no prescribed superlinear upper scale can replace linear growth as a universal sufficient condition for finite-dimensionality.
\end{mainthm}

Theorem~\ref{mainthm:sharp-threshold} is proved as Theorem~\ref{thm:slow-universality} in Section~\ref{sec:switch}.

Theorem~\ref{thm:marking-invariance} shows that the coarse linear-rescaling growth type of $R_L$, and in particular the condition $R_L(n)=O(n)$, is independent of the free basis and, more generally, of the finite marked graph.  In the geodesic case this result recovers the corresponding growth-invariance result for ordinary algebraic laminations from~\cite[Proposition~5.7 and Corollary~5.8]{Kap26}: when every leaf has exactly two endpoints, its hull in a marked tree is a bi-infinite geodesic, and a radius-$n$ round pattern is just a centered geodesic segment of length $2n$.  Thus, up to the linear reindexing $n\mapsto2n$ and the harmless bounded factor coming from orientation, $R_L^\Gamma(n)$ records the usual laminary-language complexity.  Remark~\ref{rem:outer-space-round} gives the analogous metric marked-graph formulation.

The paper naturally divides into three parts.  Section~2 collects preliminaries.  Sections~3--5 develop the symbolic tree coding and transport the basic $F_4$ construction by malnormal quasiconvex induction; Section~6 proves Choquet universality.  Sections~7--8 study raw and switch complexity, proving Theorems~\ref{mainthm:complexity} and~\ref{mainthm:sharp-threshold} together with the exact inverse-limit dimension formula.  Section~9 proves change-of-marking invariance for the coarse growth type of round complexity.

\section{Preliminaries}\label{sec:prelim}

We fix terminology and recall the needed background.  For subset currents on free groups see Kapovich--Nagnibeda~\cite{KN13} and Kapovich~\cite{Kap17}; for hyperbolic and surface groups see Sasaki~\cite{Sas22}.  For geodesic currents see Bonahon~\cite{Bon88,Bon91}, and for free-group algebraic laminations see Coulbois--Hilion--Lustig~\cite{CHL08}.

\subsection{Subset laminations and convex-hull realizations}\label{subsec:subset-laminations}

\begin{conv}[Hyperbolic-group boundary notation]\label{conv:hyperbolic-boundary}
Throughout Subsections~\ref{subsec:subset-laminations} and~\ref{subsec:subset-currents}, $G$ denotes a non-elementary word-hyperbolic group, $\partial G$ its Gromov boundary, and
\[
 \calH(\partial G)=\{S\subseteq\partial G:S\text{ is closed and }\#S\ge2\},
 \qquad
 \dd G=\{S\in\calH(\partial G):\#S=2\}.
\]
The hyperspace $\calH(\partial G)$ carries the Vietoris topology, equivalently the Hausdorff topology for any compatible metric on $\partial G$.  It is locally compact, separable, and metrizable, and the action of $G$ on $\partial G$ induces a continuous action on it; see~\cite[Theorem~2.2]{Sas22}.
\end{conv}

\begin{defn}[Subset laminations]\label{def:subset-lamination}
\begin{enumerate}[(1)]
\item A \emph{subset algebraic lamination} on $G$, or simply a \emph{subset lamination}, is a closed $G$-invariant subset
\[
 L\subseteq\calH(\partial G).
\]
Its elements are called \emph{leaves}.
\item A subset lamination $L$ is \emph{geodesic} if $L\subseteq\dd G$, \emph{branching} if $L\not\subseteq\dd G$, and \emph{genuinely branching} if
\[
 L\cap\dd G=\varnothing.
\]
A leaf $S\in L$ with $\#S\ge3$ is a \emph{branching leaf}.
\item A nonempty subset lamination $L$ is \emph{minimal} if it has no proper nonempty closed $G$-invariant subset.
\end{enumerate}
\end{defn}

The term ``subset algebraic lamination'' comes from~\cite[Section~10.7]{KN13}.  Geodesic subset laminations are ordinary algebraic laminations viewed as unordered boundary pairs; see~\cite{CHL08} for the ordered-pair, symbolic, and laminary-language formulations.  Nonminimal subset laminations may contain both geodesic and branching leaves.

If $L$ is minimal, then the closed invariant subset $L\cap\dd G$ is either empty or all of $L$.  Hence every minimal branching subset lamination is automatically genuinely branching.

For $F=F(A)$ with Cayley tree $X_A=\Cay(F,A)$, define the \emph{Cayley-tree}, or \emph{convex-hull}, realization by
\[
 \operatorname{Hull}_{X_A}(L)
 :=\{\Hull_{X_A}(S):S\in L\}.
\]
Its geometric leaves are leafless infinite subtrees: two-point leaves give bi-infinite geodesics and branching leaves give branching subtrees.  The boundary--hull correspondence is the equivariant bijection recalled in Subsection~\ref{subsec:free-rooted}.

Similarly, if $G=\pi_1(\Sigma)$ for a closed surface equipped with a hyperbolic metric and we identify $\partial G$ with $\partial_\infty\mathbb H^2$, the \emph{hyperbolic convex-hull realization} is
\[
 \operatorname{Hull}_{\mathbb H^2}(L)
 :=\{\Hull_{\mathbb H^2}(S):S\in L\}.
\]
For $\#S=2$ the geometric leaf is a complete geodesic; for $\#S\ge3$ it is a nondegenerate closed convex subset with ideal boundary $S$ (possibly all of $\mathbb H^2$).  Distinct hulls need not be disjoint, so this family is a geometric realization rather than a classical surface lamination.

\subsection{Subset currents on hyperbolic groups}\label{subsec:subset-currents}

Convention~\ref{conv:hyperbolic-boundary} remains in force.

\begin{defn}[Subset currents and carried-current cones]\label{def:subset-current}\label{def:carried-current-cone}
\begin{enumerate}[(1)]
\item A \emph{subset current} on $G$ is a $G$-invariant locally finite positive Borel measure on $\calH(\partial G)$.  The cone of subset currents, endowed with the weak-* topology, is denoted $\SC(G)$.  Here local finiteness means finiteness on compact subsets.
\item If $L$ is a subset lamination, its \emph{carried-current cone} is
\[
 \SC_G(L):=\{\mu\in\SC(G):\supp(\mu)\subseteq L\}.
\]
We say that $L$ \emph{carries} a current $\mu$ when $\mu\in\SC_G(L)$.  If $\SC_G(L)\ne\{0\}$, put
\[
 V_G(L):=\operatorname{span}_{\R}\SC_G(L),
 \qquad
 d_G(L):=\dim_{\R}V_G(L)\in\N\cup\{\infty\}.
\]
When the ambient group is clear, we use the abbreviations $\SC(L)$, $V(L)$, and $d(L)$.
\item A \emph{compact convex base} of a nonzero cone $\SC_G(L)$ is a compact convex subset of $\SC_G(L)\setminus\{0\}$ meeting every nonzero ray in exactly one point.  A nonzero carried current is \emph{projectively ergodic} if its ray is an extreme ray of $\SC_G(L)$, equivalently if its representative in any compact convex base is an extreme point.  We also call such extreme rays \emph{ergodic rays}.  We say that $L$ \emph{carries a projectively unique nonzero current} if $\SC_G(L)$ is a single nonzero ray.
\end{enumerate}
\end{defn}

\begin{defn}[Support-unique ergodicity]\label{def:SUE}
Let $0\ne\mu\in\SC(G)$.  We say that $\mu$ is \emph{support-uniquely ergodic} if
\[
 \{\nu\in\SC(G):\supp(\nu)\subseteq\supp(\mu)\}
   =\R_{\ge0}\mu.
\]
Equivalently, the subset lamination $\supp(\mu)$ carries a projectively unique nonzero subset current.  We use the qualifier ``support'' to distinguish this property from ordinary ergodicity, which refers to extremality of an invariant measure or carried-current ray.
\end{defn}

See~\cite{KN13} and~\cite[Sections~2--4]{Sas22} for the general theory.  The support of every nonzero current is a nonempty subset lamination; if $L$ is minimal and $0\ne\mu\in\SC_G(L)$, then $\supp(\mu)=L$.  Hence for minimal $L$, carrying a projectively unique nonzero current is equivalent to support-unique ergodicity of every nonzero carried current.

Identify flip-invariant measures on ordered pairs with measures on unordered pairs by taking half the push-forward under the quotient map.  Extension by zero from $\dd G$ to $\calH(\partial G)$ then gives a natural linear embedding
\[
 \GC(G)\hookrightarrow\SC(G).
\]
For the classical surface theory see~\cite{Bon88}, and for geodesic currents on negatively curved groups see~\cite{Bon91}.

Counting currents are the basic atomic examples.  For an infinite quasiconvex $H\le G$, the normalized atomic measure on the $G$-orbit of $\Lambda H$ is locally finite under the standard hypotheses; see~\cite[Theorem~2.8]{Sas22}.  For free groups see~\cite[Definition~4.3 and Lemma~4.4]{KN13}, and for closed surface groups~\cite[Lemma~2.7]{Sas22}.  Thus every nontrivial finitely generated subgroup in either setting defines $\eta_H$.

\begin{defn}[Rational subset currents]\label{def:rational-current}
Let $G$ be a nonabelian free group or a closed hyperbolic surface group.  A subset current $\rho\in\SC(G)$ is \emph{rational} if
\[
 \rho=c\eta_H
\]
for some $c\ge0$ and some nontrivial finitely generated subgroup $H\le G$.
\end{defn}

Rational subset currents are dense in $\SC(F_N)$ by~\cite[Theorem~5.8]{KN13} (see also~\cite[Theorem~5.1]{Kap17}), and in the closed-surface case by~\cite[Theorem~8.21]{Sas22}.

\begin{lem}[Rational currents have countable support]\label{lem:rational-countable}
Let $G$ be a nonabelian finitely generated free group or a closed hyperbolic surface group.  If $0\ne\rho\in\SC(G)$ is rational, then $\supp(\rho)$ is countable.
\end{lem}

\begin{proof}
The atoms of a nonzero rational current lie in a countable orbit and have a common positive mass.  Local finiteness makes that orbit closed and discrete in $\calH(\partial G)$, so it equals the support.
\end{proof}

\subsection{Free groups: cylinders, weights, and rooted trees}\label{subsec:free-rooted}

\begin{conv}[Free-group rooted-tree notation]\label{conv:free-rooted}
Throughout Subsection~\ref{subsec:free-rooted}, $F=F(A)$ denotes a finitely generated nonabelian free group with fixed free basis $A$, and
\[
 X=\Cay(F,A)
\]
denotes its standard Cayley tree.  All hulls, cylinders, and rooted-tree spaces in this subsection are taken with respect to this marking unless explicitly stated otherwise.
\end{conv}

For every finite nondegenerate subtree $K\subseteq X$, Kapovich--Nagnibeda define a compact-open \emph{subset cylinder} $\operatorname{SCyl}_A(K)\subseteq\calH(\partial F)$; these cylinders form a basis~\cite[Definition~3.2 and Proposition~3.3]{KN13}.  For a leafless tree $T$, one has $\partial T\in\operatorname{SCyl}_A(K)$ precisely when $K\subseteq T$ and every nonterminal vertex of $K$ has the same incident edges in $K$ and $T$.

\begin{defn}[Finite-subtree weights]\label{def:finite-subtree-weight}
For $\mu\in\SC(F)$ and a finite nondegenerate subtree $K\subseteq X$, the \emph{$K$-weight} of $\mu$ is
\[
 \langle K,\mu\rangle_A:=\mu\bigl(\operatorname{SCyl}_A(K)\bigr).
\]
\end{defn}

Translation invariance makes the $K$-weight depend only on the $F$-translation class of $K$, and the collection of all finite-subtree weights determines $\mu$.  Finite additivity of cylinders gives the Kirchhoff-type splitting identities of~\cite[Proposition~3.11]{KN13}.

For later use, one may organize these weights by \emph{round graphs}: finite subtrees whose terminal vertices all lie at a fixed radius from a chosen center.  Grade-$n$ round weights determine all weights of radius at most $n$ by the splitting formulas, while \emph{semi-round graphs} encode the corresponding switch equations; see~\cite[Sections~3--4]{Kap17}.  Section~\ref{sec:switch} will use the support-restricted versions of this finite-dimensional linear system.

Let $\calT(X)$ denote the space of infinite subtrees $T\subseteq X$ with no vertices of degree one.  Kapovich--Nagnibeda prove that
\[
 S\longmapsto \Hull_X(S)
\]
is an $F$-equivariant bijection from $\calH(\partial F)$ to $\calT(X)$, with inverse $T\mapsto\partial T$; see~\cite[Proposition~4.16]{KN13}.  With the local topology on rooted trees, it restricts to a homeomorphism
\[
 \{S\in\calH(\partial F):1\in\Hull_X(S)\}\longrightarrow\calT_1(X),
\]
since a finite rooted hull ball is determined by which elements of a finite clopen partition of $\partial F$ by boundary shadows meet $S$.

\begin{defn}[Rooted-tree transversal and root-change invariant measures]\label{def:root-change-measures}
Put
\[
 \calT_1(X)=\{T\in\calT(X):1\in T\}.
\]
For $g\in F$ let
\[
 \calT_{1,g}(X)=\{T\in\calT_1(X):g\in T\}
\]
and define the partial root-change homeomorphism
\[
 r_g:\calT_{1,g}(X)\to\calT_{1,g^{-1}}(X),\qquad r_g(T)=g^{-1}T.
\]
The cone of finite positive Borel measures on $\calT_1(X)$ invariant under every partial map $r_g$ is denoted $\calM_1(X)$; its elements are called \emph{root-change invariant measures}.
\end{defn}

The space $\calT_1(X)$ is compact and totally disconnected.  The basic rooted-tree description is:

\begin{thm}[Kapovich--Nagnibeda]\label{thm:KN-rooted}
There is a canonical $\R_{\ge0}$-linear homeomorphism
\[
 t:\calM_1(X)\longrightarrow\SC(F).
\]
For every $\nu\in\calM_1(X)$ and every finite nondegenerate subtree $K_0\subseteq X$ containing $1$, the weight of $K_0$ in $t(\nu)$ is the $\nu$-measure of the corresponding rooted-tree cylinder.
\end{thm}

\begin{proof}
The preceding theorem is Proposition~5.3 of~\cite{KN13}.  We will use both the statement and its cylinder-set proof.
\end{proof}

\begin{defn}[Normalized carried-current base]\label{def:normalized-current-base}
Let $L\subseteq\calH(\partial F)$ be a subset lamination carrying a nonzero subset current.  Define
\[
 \mathcal M_A(L):=
 \{\mu\in\SC_F(L):t^{-1}(\mu)(\calT_1(X))=1\}.
\]
We call $\mathcal M_A(L)$ the \emph{normalized carried-current base} associated to the marking $A$ and the rooted-tree normalization.
\end{defn}

The rooted transversal of $L$ is closed in the compact space $\calT_1(X)$ by the rooted boundary--hull homeomorphism above.  Thus the root-change invariant probability measures supported on it form a compact convex set, so $\mathcal M_A(L)$ is indeed a compact convex base of $\SC_F(L)$.

We record the corresponding support statement.

\begin{defn}[Boundary saturation]\label{def:saturation}
Let $K\subseteq\calT_1(X)$ be a closed root-change-invariant subset.  Define its \emph{boundary saturation}
\[
 \Sat(K)=\left\{S\in\calH(\partial F):
  g^{-1}\Hull_X(S)\in K \text{ for some vertex }g\in\Hull_X(S)\right\}.
\]
Because $K$ is invariant under root change, ``for some vertex'' is equivalent to ``for every vertex''.
\end{defn}

\begin{lem}[Support transfer]\label{lem:support-transfer}
Let $0\ne\nu\in\calM_1(X)$ and put $K=\supp(\nu)$ and $\mu=t(\nu)$.  Then
\[
 \supp(\mu)=\Sat(K).
\]
Moreover, if $\nu'\in\calM_1(X)$ and
\[
 \supp(t(\nu'))\subseteq\supp(t(\nu)),
\]
then
\[
 \supp(\nu')\subseteq K.
\]
\end{lem}

\begin{proof}
The support $K$ is root-change invariant.  Indeed, if $T\in K$ and $g\in T$, then $\calT_{1,g}(X)$ is an open neighborhood of $T$ in its domain and $r_g$ is a homeomorphism there.  Root-change invariance of $\nu$ implies that every neighborhood of $r_g(T)$ has positive measure.

For a rooted-tree cylinder $C$ and its corresponding subset cylinder $U$, Theorem~\ref{thm:KN-rooted} gives
\[
 t(\nu)(U)=\nu(C).
\]
These cylinders give the corresponding neighborhood bases.  If $S\in\Sat(K)$, translate so that
\[
 Y:=\Hull_X(S)\in K\subseteq\calT_1(X).
\]
Since $Y\in\supp(\nu)$, every rooted cylinder containing $Y$ has positive $\nu$-measure, hence every corresponding subset cylinder has positive $\mu$-measure.  Thus $S\in\supp(\mu)$.

Conversely, suppose $S\notin\Sat(K)$.  Choose a vertex $g\in\Hull_X(S)$ and put
\[
 Y=g^{-1}\Hull_X(S)\in\calT_1(X).
\]
Then $Y\notin K$.  Since $K$ is closed and rooted-tree cylinders form a clopen basis, there is a rooted cylinder $C$ with $Y\in C$ and $C\cap K=\varnothing$.  Let $U$ be the corresponding subset cylinder.  Then
\[
 t(\nu)(U)=\nu(C)=0.
\]
Thus $U$ is disjoint from $\supp(\mu)$, and after translating back by $g$ we obtain an open neighborhood of $S$ of $\mu$-measure zero.  Therefore $S\notin\supp(\mu)$.

The final assertion is immediate if $\nu'=0$.  Otherwise let $K'=\supp(\nu')$; the first part gives
\[
 \supp(t(\nu'))=\Sat(K').
\]
If $Y\in K'$, then $\partial Y\in\Sat(K')\subseteq\Sat(K)$.  Since $1$ is a vertex of $Y=\Hull_X(\partial Y)$, Definition~\ref{def:saturation} implies $Y\in K$.  Thus $K'\subseteq K$.
\end{proof}

\begin{cor}\label{cor:root-to-current-UE}
Let $K\subseteq\calT_1(X)$ be a closed root-change-invariant set that carries, up to scalar multiple, a unique nonzero finite root-change-invariant measure $\nu$.  Assume that $\supp(\nu)=K$.  Then $\mu=t(\nu)$ is support-uniquely ergodic.
\end{cor}

\begin{proof}
Suppose $\rho\in\SC(F)$ and $\supp(\rho)\subseteq\supp(\mu)$.  Write $\rho=t(\nu')$ using Theorem~\ref{thm:KN-rooted}.  Lemma~\ref{lem:support-transfer} gives $\supp(\nu')\subseteq K$.  By the assumed uniqueness on $K$, $\nu'=c\nu$ for some $c\ge0$, and therefore $\rho=c\mu$.
\end{proof}

\subsection{Invariant measures and Choquet simplices}\label{subsec:choquet}

If a countable group $\Gamma$ acts continuously on a compact metrizable $Y$, then, whenever invariant probability measures exist, $\mathcal M_\Gamma(Y)$ is a compact metrizable convex set whose extreme points are the ergodic measures; ergodic decomposition makes it a Choquet simplex.  See Phelps~\cite{Phe01} and~\cite[Theorem~3.3]{KM04}.

\begin{defn}[Choquet simplex]\label{def:choquet-simplex}
A compact convex metrizable set $\mathcal K$ is a \emph{Choquet simplex} if every point of $\mathcal K$ has a unique representing probability measure on the extreme boundary in the sense of Choquet theory.
\end{defn}

A finite-dimensional Choquet simplex is an ordinary finite simplex.  Unique ergodicity means that $\mathcal M_\Gamma(Y)$ is a singleton, and minimality forces every invariant probability measure to have full support.

We use standard Toeplitz terminology~\cite{Dow91,Que10,CCG24}: a two-sided sequence is Toeplitz if every coordinate is periodic with some nonzero period, and its orbit closure is a Toeplitz subshift.

A central external input for Section~\ref{sec:universality} is Downarowicz's realization theorem~\cite[Theorem~1]{Dow91}: every nonempty metrizable Choquet simplex is affinely homeomorphic to the invariant-measure simplex of a minimal binary Toeplitz $\mathbb Z$-subshift.  Later we also use low-complexity realization results for such simplices and standard ergodic-decomposition facts to identify extreme points with projectively ergodic carried currents.

\subsection{Quasiconvex subgroups and induction}\label{subsec:prelim-induction}

For an infinite quasiconvex $H\le G$, inclusion extends to an $H$-equivariant embedding $\partial H\hookrightarrow\partial G$ with image $\Lambda H$.  We use the standard limit-set intersection property~\cite[Theorem~1.1]{Yan12} and quasiconvexity of the corresponding cosets; see~\cite{BH99}.

Sasaki~\cite[Theorem~4.3]{Sas22} constructs for such an inclusion a continuous linear induction map
\[
 \iota_H^G:\SC(H)\longrightarrow\SC(G),
\]
obtained by summing translates over $G/H$.  For malnormal $H$, distinct translated limit sets are disjoint, allowing restriction to recover the $H$-current; Section~5 gives the precise statement.  Kielak--Linton~\cite[Theorem~1.11]{KL24} provide malnormal quasiconvex free subgroups of any prescribed finite rank.

\section{Encoding an \texorpdfstring{$F_2$}{F2}-subshift by branching trees}

\subsection{The symbolic system}

\begin{conv}[Tree-coding setup]\label{conv:tree-coding}
Throughout Section~3,
\[
 B=F(u,v),\qquad F=F(a_0,a_1,b_0,b_1),
\]
with standard Cayley trees
\[
 X_B=\Cay(B,\{u,v\}),\qquad
 X=\Cay(F,\{a_0,a_1,b_0,b_1\}).
\]
We use the right shift on $\{0,1\}^{B}$ given by
\begin{equation}
 (x\cdot w)(z)=x(wz),\qquad x\in\{0,1\}^{B},\quad w,z\in B,
 \label{eq:right-shift}
\end{equation}
so $(x\cdot w_1)\cdot w_2=x\cdot(w_1w_2)$.  Unless a statement explicitly imposes stronger hypotheses, $\Omega\subseteq\{0,1\}^{B}$ denotes a nonempty closed $B$-invariant subshift.  The objects $\kappa_x,T_x,\Phi$, and $K_\Omega$ are defined below from the preceding data.
\end{conv}

If $\Omega$ is minimal, every $B$-invariant probability measure $\lambda$ has full support:
\begin{equation}
 \supp(\lambda)=\Omega.
 \label{eq:lambda-fullsupport}
\end{equation}
If $\Omega$ is infinite and minimal, then it has no isolated points and hence is a Cantor space.

\subsection{The tree coding}

Convention~\ref{conv:tree-coding} remains in force.  For each $x\in\Omega$, define a graph map
\[
 \kappa_x:X_B\longrightarrow X
\]
rooted at $1$ as follows.  Set $\kappa_x(1)=1$.  For every $w\in B$, map the positively oriented $u$-edge
\[
 w\longrightarrow wu
\]
to the edge based at $\kappa_x(w)$ labelled $a_{x(w)}$, and map the positively oriented $v$-edge
\[
 w\longrightarrow wv
\]
to the edge based at $\kappa_x(w)$ labelled $b_{x(w)}$.  Since $X_B$ is a tree, these prescriptions determine $\kappa_x$ uniquely.

\begin{lem}\label{lem:embedding}
For every $x\in\Omega$, the map $\kappa_x:X_B\to X$ is an embedding.  Its image
\[
 T_x:=\kappa_x(X_B)
\]
is an exactly $4$-regular subtree of $X$ containing $1$.
\end{lem}

\begin{proof}
At a vertex $w\in B$, the four oriented incident edges in $X_B$ are sent to edges of $X$ with labels
\begin{equation}
 a_{x(w)},\qquad b_{x(w)},\qquad
 a_{x(wu^{-1})}^{-1},\qquad b_{x(wv^{-1})}^{-1}.
 \label{eq:four-labels}
\end{equation}
The four labels in~\eqref{eq:four-labels} are distinct elements of the oriented generating alphabet
\[
 \{a_0^{\pm1},a_1^{\pm1},b_0^{\pm1},b_1^{\pm1}\}.
\]
Hence $\kappa_x$ is locally injective, and therefore injective because both graphs are trees.  Since $X_B$ is $4$-regular, so is $T_x$.
\end{proof}

In particular $T_x\in\calT_1(X)$.  Define
\[
 \Phi:\Omega\to\calT_1(X),\qquad \Phi(x)=T_x,
\]
and put
\[
 K_\Omega:=\Phi(\Omega).
\]

\begin{lem}\label{lem:Phi-homeo}
The map $\Phi$ is a homeomorphism from $\Omega$ onto the compact set $K_\Omega$.
\end{lem}

\begin{proof}
The radius-$R$ ball in $T_x$ depends on finitely many coordinates of $x$, so $\Phi$ is continuous.

It is also injective.  From the rooted labelled subtree $T_x$ one reconstructs $x$ and the embedded copy of $X_B$ recursively.  At the root, exactly one of the positive $a$-edges $a_0,a_1$ is present, and its index is $x(1)$; similarly there is exactly one positive $b$-edge.  At any already reconstructed vertex $\kappa_x(w)$, the unique positive $a$-edge identifies the $u$-neighbor and records $x(w)$, the unique positive $b$-edge identifies the $v$-neighbor, and the corresponding negative $a$- and $b$-edges identify the $u^{-1}$- and $v^{-1}$-neighbors.  Induction over word length recovers $x(w)$ for every $w\in B$.  Hence $T_x=T_y$ implies $x=y$.

Compactness of $\Omega$ and Hausdorffness of $\calT_1(X)$ now make $\Phi$ a homeomorphism onto the compact set $K_\Omega$.
\end{proof}

The next identity is the core of the construction.

\begin{lem}[Root change equals shift]\label{lem:root-shift}
For all $x\in\Omega$ and $w\in B$,
\begin{equation}
 \kappa_x(w)^{-1}T_x=T_{x\cdot w}.
 \label{eq:root-shift}
\end{equation}
More precisely, for every $z\in B$,
\begin{equation}
 \kappa_{x\cdot w}(z)=\kappa_x(w)^{-1}\kappa_x(wz).
 \label{eq:kappa-cocycle}
\end{equation}
\end{lem}

\begin{proof}
Consider the map
\[
 z\longmapsto \kappa_x(w)^{-1}\kappa_x(wz)
\]
from $X_B$ to $X$.  It sends $1$ to $1$.  At a vertex $z$, the positively oriented $u$-edge is labelled
\[
 a_{x(wz)}=a_{(x\cdot w)(z)},
\]
and the positively oriented $v$-edge is labelled
\[
 b_{x(wz)}=b_{(x\cdot w)(z)}.
\]
By uniqueness of the rooted graph map with the prescribed labels, the displayed map is exactly $\kappa_{x\cdot w}$.  Hence~\eqref{eq:kappa-cocycle} holds, and taking images gives~\eqref{eq:root-shift}.
\end{proof}

$K_\Omega$ is therefore root-change invariant.

\subsection{Invariant measures: the groupoid correspondence}

The root-change identity also identifies the invariant-measure theories.

\begin{prop}\label{prop:measure-correspondence}
The push-forward map
\[
 \Phi_*:\{\text{finite $B$-invariant Borel measures on $\Omega$}\}
 \longrightarrow
 \{\nu\in\calM_1(X):\supp(\nu)\subseteq K_\Omega\}
\]
is an $\R_{\ge0}$-linear bijection.
\end{prop}

\begin{proof}
For $x\in\Omega$ and $w\in B$, every vertex of $T_x$ has the form $\kappa_x(w)$ for a unique $w$.  Hence
\[
 (x,w)\longmapsto (T_x,\kappa_x(w))
\]
identifies the transformation groupoid of the $B$-action on $\Omega$ with the root-change groupoid restricted to $K_\Omega$, by Lemma~\ref{lem:root-shift}.  Therefore push-forward by $\Phi$ identifies their finite invariant measures.

For fixed $g\in F$, the domain $\{x:g\in T_x\}$ splits into finitely many clopen pieces on which $g=\kappa_x(w)$ for fixed $w$ and $r_g$ corresponds to $x\mapsto x\cdot w$.  Indeed, $\kappa_x$ is an edge-preserving rooted embedding, so $|\kappa_x(w)|_A=|w|_{\{u,v\}}$; hence $g=\kappa_x(w)$ forces $|w|_{\{u,v\}}=|g|_A$, leaving only finitely many possible $w$.  For each such $w$, the condition $g=\kappa_x(w)$ depends on only finitely many coordinates of $x$ along the geodesic from $1$ to $w$, and is therefore clopen.  Conversely, for fixed $w$, the value of $\kappa_x(w)$ is determined by finitely many binary choices and hence takes only finitely many values.  Thus $B$-shift invariance is equivalent to invariance under all partial root changes, and $\Phi$ gives the asserted linear bijection.
\end{proof}

For the uniquely ergodic specialization, suppose in addition that $\Omega$ is infinite and minimal and has a unique $B$-invariant probability measure $\lambda$.

\begin{cor}\label{cor:K-unique}
Suppose that $\Omega$ is infinite, minimal, and uniquely ergodic, and let $\lambda$ be its unique $B$-invariant probability measure.  Then the probability measure
\[
 \nu_\Omega:=\Phi_*\lambda
\]
is root-change invariant, has support $K_\Omega$, and is, up to nonnegative scalar multiple, the unique finite root-change-invariant measure supported on $K_\Omega$.
\end{cor}

\begin{proof}
Root-change invariance and uniqueness follow from Proposition~\ref{prop:measure-correspondence} and unique ergodicity of $\Omega$.  Equation~\eqref{eq:lambda-fullsupport} and the fact that $\Phi$ is a homeomorphism imply
\[
 \supp(\nu_\Omega)=K_\Omega.
\]
\end{proof}

\subsection{The resulting subset current on $F_4$}

Define
\begin{equation}
 \mu_\Omega:=t(\nu_\Omega)\in\SC(F),
 \qquad F=F(a_0,a_1,b_0,b_1).
 \label{eq:muOmega}
\end{equation}

\begin{thm}\label{thm:symbolic}
Let $B=F(u,v)$ and let $\Omega\subseteq\{0,1\}^{B}$ be an infinite minimal uniquely ergodic subshift for the right shift~\eqref{eq:right-shift}.  Form the tree-coding current $\mu_\Omega\in\SC(F_4)$ as in Convention~\ref{conv:tree-coding} and~\eqref{eq:muOmega}.  Then $\mu_\Omega$ has the following properties.
\begin{enumerate}[(1)]
\item $\mu_\Omega$ is support-uniquely ergodic.
\item Every $S\in\supp(\mu_\Omega)$ is a Cantor set, and its convex hull in the Cayley tree of $F_4$ is exactly $4$-regular.  In particular
\[
 \supp(\mu_\Omega)\cap\dd F_4=\varnothing.
\]
\item $\supp(\mu_\Omega)$ is uncountable.  Consequently $\mu_\Omega$ is not rational.
\end{enumerate}
\end{thm}

\begin{proof}
Part~(1) follows immediately from Corollaries~\ref{cor:root-to-current-UE} and~\ref{cor:K-unique}.

For part~(2), Lemma~\ref{lem:support-transfer} gives
\[
 \supp(\mu_\Omega)=\Sat(K_\Omega).
\]
If $S\in\Sat(K_\Omega)$ and $g$ is any vertex of $\Hull_X(S)$, then
\[
 g^{-1}\Hull_X(S)\in K_\Omega.
\]
By Lemma~\ref{lem:embedding}, $\Hull_X(S)$ is exactly $4$-regular, so its boundary is a Cantor set.  It is not a two-point set, whose convex hull would be a bi-infinite line.

For part~(3), the map
\[
 \Omega\longrightarrow\supp(\mu_\Omega),\qquad x\longmapsto\partial T_x,
\]
is injective: if $\partial T_x=\partial T_y$, then the convex-hull correspondence gives $T_x=T_y$, and Lemma~\ref{lem:Phi-homeo} gives $x=y$.  Since an infinite minimal subshift is a Cantor space, $\supp(\mu_\Omega)$ is uncountable; Lemma~\ref{lem:rational-countable} then gives non-rationality.
\end{proof}

\begin{rem}\label{rem:atomless}
Under Corollary~\ref{cor:K-unique}, $\lambda$ is nonatomic: an atom would force a finite orbit.  Hence $\nu_\Omega$ is nonatomic, so $\mu_\Omega$ is genuinely measured rather than a disguised counting current.
\end{rem}

\section{Two choices of the symbolic input}

\subsection{An elementary construction factoring through $\Z$}

Let $Y\subseteq\{0,1\}^{\Z}$ be a nonempty subshift, with $(\sigma y)(k)=y(k+1)$.  Define
\begin{equation}
 \pi:B=F(u,v)\longrightarrow\Z,
 \qquad \pi(u)=1,\quad \pi(v)=0.
 \label{eq:p-to-Z}
\end{equation}
For $y\in Y$ put
\begin{equation}
 x_y(w)=y\bigl(\pi(w)\bigr)\quad(w\in B),
 \qquad \Omega_Y=\{x_y:y\in Y\}.
 \label{eq:pullback-subshift}
\end{equation}

\begin{prop}\label{prop:Fibonacci}
The map $y\mapsto x_y$ is a homeomorphism from $Y$ to $\Omega_Y$ satisfying
\[
 x_y\cdot w=x_{\sigma^{\pi(w)}y}.
\]
It identifies the invariant-measure simplices of $Y$ and $\Omega_Y$.  If $Y$ is infinite and minimal, so is $\Omega_Y$; if $Y$ is uniquely ergodic, so is $\Omega_Y$.
\end{prop}

\begin{proof}
The map is continuous, and $y(k)=x_y(u^k)$ gives its inverse.  The displayed identity follows from the definition and the homomorphism property of $\pi$.  Since $\pi(B)=\Z$, the assertions about invariant measures and minimality follow.
\end{proof}

For an explicit uniquely ergodic choice, take the Fibonacci subshift $Y_{\rm Fib}$ associated to $0\mapsto01$, $1\mapsto0$, and write $\Omega_{\rm Fib}=\Omega_{Y_{\rm Fib}}$.  Its minimality, aperiodicity, and unique ergodicity are classical~\cite{Que10}.  Theorem~\ref{thm:symbolic} gives a support-uniquely-ergodic, non-rational current on $F_4$ with genuinely branching support.  Here $v$ acts trivially, so the symbolic dynamics are effectively one-dimensional.

\subsection{A genuinely nonabelian symbolic input}

For genuinely nonabelian symbolic dynamics one may instead use:

\begin{rem}[Free Toeplitz input]\label{rem:Toeplitz-input}
Cecchi Bernales--Cortez--G\'omez~\cite[Proposition~3.8 and Theorem~1.2]{CCG24} construct, for $B=F(u,v)$, a free regular binary Toeplitz $B$-subshift $\Omega_{\rm Tp}\subseteq\{0,1\}^B$ which is minimal and uniquely ergodic.  Applying Theorem~\ref{thm:symbolic} therefore gives another support-uniquely-ergodic, non-rational current on $F_4$ with genuinely branching support consisting of Cantor subsets of the boundary.  Since the $B$-action is free, it cannot factor through a proper quotient of $B$.  Equivalently, if $w\ne1$, then Lemma~\ref{lem:root-shift} gives
\[
 \kappa_x(w)^{-1}T_x=T_{x\cdot w}\ne T_x
\]
for every $x\in\Omega_{\rm Tp}$.
\end{rem}

\section{Induction through a malnormal quasiconvex free subgroup}

We transport the $F_4$ examples by an argument valid for general torsion-free hyperbolic groups.

\subsection{Quasiconvex induction of subset currents}

\begin{conv}[Quasiconvex induction setup]\label{conv:induction}
Throughout this subsection, $G$ denotes a non-elementary word-hyperbolic group and $H\le G$ an infinite quasiconvex subgroup.  We identify $\partial H$ with its image $\Lambda H\subseteq\partial G$ under the boundary embedding induced by inclusion.  The terms \emph{quasiconvex}, \emph{limit set}, and \emph{malnormal} are used in their standard hyperbolic-group sense; see Subsection~\ref{subsec:prelim-induction} and~\cite{BH99}.
\end{conv}

Sasaki~\cite[Theorem~4.3]{Sas22} defines a continuous $\R_{\ge0}$-linear map
\[
 \iota_H^G:\SC(H)\longrightarrow\SC(G)
\]
given at the measure level by
\begin{equation}
 \iota_H^G(\mu)
   =\sum_{gH\in G/H}g_*\mu,
 \label{eq:induction-sum}
\end{equation}
where $\mu$ is regarded as a measure on the hyperspace of closed subsets of $\Lambda H$.

We need the following local-finiteness fact.

\begin{lem}[Local finiteness of quasiconvex carriers]\label{lem:carrier-local-finite}
The family
\[
 \big\{\calH(g\Lambda H):gH\in G/H\big\}
\]
is locally finite in $\calH(\partial G)$.  Equivalently, every compact subset of $\calH(\partial G)$ meets only finitely many of these hyperspaces.
\end{lem}

\begin{proof}
Fix a visual metric $d$ on $\partial G$; see, for example,~\cite[Chapter~III.H]{BH99}.  If $C\subseteq\calH(\partial G)$ is compact, then
\[
 \inf_{S\in C}\diam_d(S)>0.
\]
Indeed, continuity of diameter would otherwise produce a singleton limit, contradicting compactness of $C$ inside $\calH(\partial G)$.

It therefore suffices to prove that, for every $\epsilon>0$, only finitely many cosets $gH$ satisfy
\[
 \diam_d(g\Lambda H)\ge\epsilon.
\]
Choose $\xi,\eta\in g\Lambda H$ with $d(\xi,\eta)\ge\epsilon$.  Visual-metric estimates put a geodesic joining them uniformly close to $1$, with bound depending on $\epsilon$.  Quasiconvexity puts that geodesic in a uniform neighborhood of $gH$.  Hence $gH$ meets a fixed finite ball, which meets only finitely many left cosets.
\end{proof}

We next use malnormality.  Recall that $H$ is malnormal if
\[
 H\cap gHg^{-1}=\{1\}\qquad\text{for every }g\notin H.
\]
For quasiconvex subgroups of a hyperbolic group one has the standard limit-set intersection formula
\begin{equation}
 \Lambda A\cap\Lambda B=\Lambda(A\cap B)
 \label{eq:limit-intersection}
\end{equation}
whenever $A$ and $B$ are quasiconvex; formula~\eqref{eq:limit-intersection} is the word-hyperbolic special case of~\cite[Theorem~1.1]{Yan12}.  Thus, if $H$ is malnormal and $gH\ne hH$, then
\begin{equation}
 g\Lambda H\cap h\Lambda H=\varnothing.
 \label{eq:disjoint-limitsets}
\end{equation}

\begin{prop}[Malnormal induction preserves carried-current cones]\label{prop:induction-cone}
Suppose that $H$ is malnormal.  Let
\[
 L\subseteq\calH(\Lambda H)
\]
be closed and $H$-invariant, and put
\begin{equation}
 \widehat L:=\bigsqcup_{gH\in G/H}gL\subseteq\calH(\partial G).
 \label{eq:induced-carrier}
\end{equation}
Then $\widehat L$ is closed and $G$-invariant, and quasiconvex induction restricts to an $\R_{\ge0}$-linear bijection
\[
 \iota_H^G:\SC_H(L)\longrightarrow\SC_G(\widehat L).
\]
The inverse map is restriction to the base carrier $\calH(\Lambda H)$.  If $L$ is minimal for the $H$-action, then $\widehat L$ is minimal for the $G$-action.
\end{prop}

\begin{proof}
By~\eqref{eq:disjoint-limitsets}, the carrier hyperspaces $\calH(g\Lambda H)$ are pairwise disjoint.  By Lemma~\ref{lem:carrier-local-finite}, this family is locally finite.  Therefore the family of closed sets $gL$ is locally finite, so the disjoint union~\eqref{eq:induced-carrier} is closed.  It is plainly $G$-invariant.

If $\mu\in\SC_H(L)$, then the induction formula~\eqref{eq:induction-sum} gives
\[
 \iota_H^G(\mu)=\sum_{gH\in G/H}g_*\mu,
\]
and local finiteness implies that this current is supported on $\widehat L$.  Conversely, let $\rho\in\SC_G(\widehat L)$ and define
\[
 \rho_H:=\left.\rho\right|_{\calH(\Lambda H)}.
\]
The carrier $\calH(\Lambda H)$ is closed, hence Borel, in $\calH(\partial G)$, and restriction of a locally finite measure to it is locally finite.  Because it is $H$-invariant and all other carrier hyperspaces are disjoint from it, $\rho_H$ is an $H$-subset current supported on $L$.  For every coset $gH$, $G$-invariance of $\rho$ gives
\[
 \left.\rho\right|_{\calH(g\Lambda H)}=g_*\rho_H.
\]
Since $\rho$ is supported on the countable disjoint union of the carrier hyperspaces, countable additivity yields
\[
 \rho=\sum_{gH\in G/H}g_*\rho_H=\iota_H^G(\rho_H).
\]
Local finiteness also makes the induced support the union of the translated supports, so restriction is inverse to induction.

Assume now that $L$ is minimal.  Let $S=gS_0$ and $T=hT_0$ be points of $\widehat L$, where $S_0,T_0\in L$.  There exist $h_n\in H$ with $h_nS_0\to T_0$.  Hence
\[
 (hh_ng^{-1})S=hh_nS_0\longrightarrow hT_0=T.
\]
Thus every $G$-orbit is dense in $\widehat L$, so $\widehat L$ is minimal.
\end{proof}

\begin{cor}[Consequences of malnormal induction]\label{cor:induction-consequences}
Suppose that $H$ is malnormal.  Let $0\ne\mu\in\SC(H)$ and put $\widehat\mu=\iota_H^G(\mu)$.  Then:
\begin{enumerate}[(1)]
\item The support of the induced current is
\[
 \supp(\widehat\mu)=\bigsqcup_{gH\in G/H}g\supp(\mu).
\]
\item If $\mu$ is support-uniquely ergodic, then $\widehat\mu$ is support-uniquely ergodic.
\item If every element of $\supp(\mu)$ is a Cantor set, then every element of $\supp(\widehat\mu)$ is a Cantor set and
\[
 \supp(\widehat\mu)\cap\dd G=\varnothing.
\]
\item If $\supp(\mu)$ is uncountable, then $\supp(\widehat\mu)$ is uncountable.  If, in addition, $G$ is a nonabelian free group or a closed hyperbolic surface group, then $\widehat\mu$ is non-rational.
\end{enumerate}
\end{cor}

\begin{proof}
Part~(1) follows from Proposition~\ref{prop:induction-cone} and local finiteness, while the cone bijection in that proposition gives part~(2).  The boundary action preserves topology and cardinality, which proves the first assertions in parts~(3) and~(4); the Cantor-set conclusion in part~(3) also gives disjointness from $\dd G$.  Finally, Lemma~\ref{lem:rational-countable} gives the non-rationality assertion in part~(4).
\end{proof}

\subsection{Existence of the malnormal free subgroup}

The remaining group-theoretic input is:

\begin{thm}[Kielak--Linton]\label{thm:KL-malnormal}
Let $G$ be a non-elementary torsion-free hyperbolic group and let $\ell\ge1$.  Then $G$ contains a malnormal quasiconvex free subgroup of rank $\ell$.
\end{thm}

Theorem~\ref{thm:KL-malnormal} is the case $k=0$ of~\cite[Theorem~1.11]{KL24}; see also~\cite{Kap99} for earlier free-subgroup constructions.

\subsection{The main existence theorem}\label{subsec:proof-main}

\begin{proof}[Proof of Theorem~\ref{mainthm:main}]
The Fibonacci pullback from Section~4 and Theorem~\ref{thm:symbolic} provide a current $\mu_4\in\SC(F_4)$ having properties~(a)--(c).  For $G$ as in the theorem, Theorem~\ref{thm:KL-malnormal} gives a malnormal quasiconvex subgroup $H\cong F_4$.  Transport $\mu_4$ to $H$ and induce it to $G$.  Corollary~\ref{cor:induction-consequences} preserves support-unique ergodicity, Cantor leaves, genuine branching, and uncountability, and in the free and closed-surface cases it also gives non-rationality.  Thus the induced current has all the required properties.
\end{proof}

\begin{cor}\label{cor:all-hyp}
Let $G$ be any non-elementary torsion-free hyperbolic group.  Then $G$ admits a nonzero support-uniquely-ergodic subset current $\mu$ such that every element of $\supp(\mu)$ is a Cantor subset of $\partial G$.  In particular $\supp(\mu)\cap\dd G=\varnothing$, and $\supp(\mu)$ is uncountable.
\end{cor}

\begin{proof}
The proof uses only non-elementarity, torsion-freeness, hyperbolicity, Theorem~\ref{thm:KL-malnormal}, and quasiconvex induction; only non-rationality uses the special classes of groups.
\end{proof}

\section{Universality of carried-current simplices}\label{sec:universality}

\begin{conv}[Universality coding notation]\label{conv:universality}
Throughout this section, $B=F(u,v)$ and $F_4=F(a_0,a_1,b_0,b_1)$.  For a nonempty closed $B$-invariant subshift $\Omega\subseteq\{0,1\}^{B}$, the objects $K_\Omega$, $\Sat(K_\Omega)$, and the associated carried-current maps are as in Convention~\ref{conv:tree-coding}, Definition~\ref{def:saturation}, and Proposition~\ref{prop:measure-correspondence}.
\end{conv}

The tree coding has the following full measure-theoretic form; Proposition~\ref{prop:measure-correspondence} uses only~\eqref{eq:root-shift}.

\begin{prop}[Measure-cone correspondence]\label{prop:general-measure-cone}
For every nonempty closed $B$-invariant subshift $\Omega\subseteq\{0,1\}^B$, the construction of Section~3 gives an $\R_{\ge0}$-linear homeomorphism
\begin{equation}
 \{\text{finite $B$-invariant Borel measures on $\Omega$}\}
 \longrightarrow \SC_{F_4}(\Sat(K_\Omega)).
 \label{eq:general-cone-isom}
\end{equation}
If $\Omega$ is minimal, then $\Sat(K_\Omega)$ is a minimal $F_4$-invariant subset of $\calH(\partial F_4)$.  In that case every nonzero current carried by $\Sat(K_\Omega)$ has support equal to $\Sat(K_\Omega)$.
\end{prop}

\begin{proof}
The rooted transversal of $\Sat(K_\Omega)$ is exactly $K_\Omega$.  The open sets $\{S:g\in\Hull_X(S)\}$, for $g\in F_4$, cover $\calH(\partial F_4)$.  On each such set, the map $S\mapsto g^{-1}\Hull_X(S)$ is a homeomorphism onto $\calT_1(X)$, and its inverse image of $K_\Omega$ is the intersection with $\Sat(K_\Omega)$.  Since $K_\Omega$ is closed, $\Sat(K_\Omega)$ is closed, and it is $F_4$-invariant by definition.

Compose the push-forward $\Phi_*$ of Proposition~\ref{prop:measure-correspondence} with the homeomorphism $t$ of Theorem~\ref{thm:KN-rooted}.  Lemma~\ref{lem:support-transfer} and the preceding description of the rooted transversal show that this composite identifies exactly the two cones in~\eqref{eq:general-cone-isom}.  Since $\Phi$ is a homeomorphism, both maps are linear homeomorphisms onto their images, so the composite is as well.

If $\Omega$ is minimal, the root-shift identity and density of each $B$-orbit make every $F_4$-orbit dense in $\Sat(K_\Omega)$.  Thus the saturation is minimal, and every nonzero carried current has full support there.
\end{proof}

For a minimal $B$-subshift $\Omega$ admitting an invariant probability measure, let $\mathcal M_B(\Omega)$ denote its compact simplex of invariant probability measures.  Put $A=\{a_0,a_1,b_0,b_1\}$, the standard basis of $F_4$ from Convention~\ref{conv:universality}.  By Definition~\ref{def:normalized-current-base}, $\mathcal M_A(\Sat(K_\Omega))$ is the normalized carried-current base associated to this marking.  Proposition~\ref{prop:general-measure-cone} restricts to an affine homeomorphism
\begin{equation}
 \mathcal M_B(\Omega)\cong_{\rm aff}\mathcal M_A(\Sat(K_\Omega)).
 \label{eq:normalized-cone-correspondence}
\end{equation}

\begin{thm}[Universality on $F_4$]\label{thm:F4-universality}
Let $\mathcal K$ be a nonempty metrizable Choquet simplex.  There exists an uncountable minimal subset lamination
\[
 L\subseteq\calH(\partial F_4)
\]
such that every $S\in L$ is a Cantor set whose convex hull in the standard Cayley tree of $F_4$ is exactly $4$-regular, and
\begin{equation}
 \mathcal M_A(L)\cong_{\rm aff}\mathcal K.
 \label{eq:F4-universality}
\end{equation}
Every nonzero current carried by $L$ has support exactly $L$ and is non-rational.  The extreme points of $\mathcal K$ correspond exactly to the ergodic rays of $\SC_{F_4}(L)$.
\end{thm}

\begin{proof}
Downarowicz~\cite[Theorem~1]{Dow91} proved that every nonempty metrizable Choquet simplex is affinely homeomorphic to the simplex of invariant probability measures of a minimal $0$--$1$ Toeplitz flow.  Choose such a binary minimal subshift
\[
 Y\subseteq\{0,1\}^{\Z}
\]
with invariant-measure simplex affinely homeomorphic to $\mathcal K$.  If $\mathcal K$ is a singleton and the chosen model is finite, replace it by the Fibonacci subshift; the replacement preserves the required simplex and makes the system infinite.

Use the pullback $\Omega_Y$ of~\eqref{eq:pullback-subshift}.  Proposition~\ref{prop:Fibonacci} identifies its invariant-measure simplex with that of $Y$:
\[
 \mathcal M_B(\Omega_Y)\cong_{\rm aff}\mathcal M_{\mathbb Z}(Y).
\]
Apply the tree coding of Section~3 and put
\[
 L:=\Sat(K_{\Omega_Y}).
\]
Proposition~\ref{prop:general-measure-cone} and~\eqref{eq:normalized-cone-correspondence} give~\eqref{eq:F4-universality}, and minimality of $Y$ gives minimality of $L$.

Every member of $L$ is the boundary of an exactly $4$-regular tree, hence a Cantor set.  Since $Y$ is infinite minimal and $y\mapsto\partial T_{x_y}$ is injective, $L$ is uncountable.  Minimality gives full support for every nonzero carried current, so Lemma~\ref{lem:rational-countable} gives non-rationality.  Extreme points correspond under the affine homeomorphism.
\end{proof}

\begin{proof}[Proof of Theorem~\ref{mainthm:universality}]
Let $G$ and $\mathcal K$ be as in the theorem.  By Theorem~\ref{thm:KL-malnormal}, choose a malnormal quasiconvex subgroup $H\cong F_4$, and transport to $H$ the lamination supplied by Theorem~\ref{thm:F4-universality}.  Proposition~\ref{prop:induction-cone} identifies its carried-current cone linearly with that of the induced lamination in $G$; hence the image of the normalized compact base is a compact convex base affinely homeomorphic to $\mathcal K$, with the same extreme points.  Corollary~\ref{cor:induction-consequences} and minimality give uncountability, Cantor leaves, genuine branching, full support of every nonzero carried current, and non-rationality.  These conclusions prove all parts of the theorem.
\end{proof}

\begin{cor}[Hyperbolic-group universality]\label{cor:hyperbolic-universality}
Let $G$ be any non-elementary torsion-free hyperbolic group and let $\mathcal K$ be a nonempty metrizable Choquet simplex.  Then there exists an uncountable minimal genuinely branching subset lamination $L\subseteq\calH(\partial G)$ such that the cone $\SC_G(L)$ has a compact convex base affinely homeomorphic to $\mathcal K$.
\end{cor}

\begin{proof}
The proof of Theorem~\ref{mainthm:universality} uses only Theorem~\ref{thm:KL-malnormal}, quasiconvex induction, and the $F_4$ universality theorem.  The special hypotheses on $G$ were used only for the usual non-rationality conclusion.
\end{proof}

\section{Raw finite-subtree complexity}\label{sec:complexity}

\begin{conv}[Free-group complexity notation]\label{conv:complexity}
Throughout Sections~\ref{sec:complexity} and~\ref{sec:switch}, $F_N=F(A)$ denotes a nonabelian free group of rank $N\ge2$ with fixed free basis $A$, and $X_A=\Cay(F_N,A)$.  For a subset lamination $L$ carrying a nonzero subset current, the abbreviations $\SC(L)$, $V(L)$, and $d(L)$ are as in Definition~\ref{def:carried-current-cone}.  All finite-subtree cylinders and weights use the marking $A$ unless otherwise specified.
\end{conv}

\begin{defn}[Finite-subtree complexity]\label{def:subtree-complexity}
For a nonempty subset lamination $L$ carrying a nonzero subset current and $n\ge1$, put
\[
 p_L(n):=\#\{[K]_{F_N}:|E(K)|=n,\ \operatorname{SCyl}_A(K)\cap L\ne\varnothing\}.
\]
Here $K\subseteq X_A$ ranges over finite nondegenerate subtrees, and $|E(K)|$ counts unoriented edges.  If $L$ is minimal, $p_L(n)$ is equivalently the number of $n$-edge translation classes having positive weight for any nonzero current carried by $L$.
\end{defn}

The pullback examples show that the natural function $p_L$ is too coarse to detect support-unique ergodicity.

\begin{defn}[Word complexity]\label{def:word-complexity}
For a subshift $Y\subseteq\mathcal A^{\Z}$ over a finite alphabet $\mathcal A$, its \emph{word-complexity function} is
\[
 P_Y(\ell):=\#\{\text{words of length $\ell$ occurring in }Y\},
 \qquad \ell\ge1.
\]
\end{defn}

\begin{conv}[Pullback complexity family]\label{conv:pullback-complexity}
Let $Y\subseteq\{0,1\}^{\Z}$ be an infinite minimal subshift, with word-complexity function $P_Y$ as in Definition~\ref{def:word-complexity}.  Use the pullback $\Omega_Y$ of~\eqref{eq:pullback-subshift} and the tree coding of Section~3 in $F_4$, and write $L_Y=\Sat(K_{\Omega_Y})$.  When $Y=Y_{\rm Fib}$ is the Fibonacci subshift, write $P_{\rm Fib}:=P_{Y_{\rm Fib}}$ and $L_{\rm Fib}:=L_{Y_{\rm Fib}}$.  Put $B=F(u,v)$ and $X_B=\Cay(B,\{u,v\})$.  For $m\ge1$, let $\mathcal C_m$ be the set of $B$-translation classes of connected $m$-vertex subtrees of $X_B$, and for such a core put
\[
 \rho(C)=\max_{w\in C}\pi(w)-\min_{w\in C}\pi(w).
\]
Let $c_m=|\mathcal C_m|$, with $c_1=1$.  By Proposition~\ref{prop:general-measure-cone}, $L_Y$ is minimal and its normalized carried-current base is affinely homeomorphic to the invariant-measure simplex of $Y$.
\end{conv}

To count cores, let $f=z(1+f)^3$ be the generating function for a branch across a prescribed oriented edge.  Then $(1+f)^4$ counts finite subtrees containing a fixed root by number of edges.  Lagrange inversion, followed by division by $m$ because an $m$-vertex translation class has exactly $m$ translates containing the root, gives, for $m\ge2$,
\begin{equation}
 c_m=\frac{4}{m(m-1)}\binom{3m}{m-2}
 \asymp m^{-5/2}\left(\frac{27}{4}\right)^m.
 \label{eq:core-count}
\end{equation}

Apart from single-edge patterns, every finite subtree whose cylinder meets $L_Y$ has a nonempty core $C$ consisting of its nonterminal vertices.  This core is connected, and each of its vertices is locally saturated in the ambient exactly $4$-regular tree.  If $C$ has $m$ vertices, it has $m-1$ internal edges, and saturation adds
\[
 4m-2(m-1)=2m+2
\]
terminal edges.  Thus the saturated pattern has $3m+1$ edges.  Conversely, it is obtained from its core by this saturation procedure.  For a fixed core $C$, the labels on the saturated pattern are determined exactly by the coordinates from one level below $\min\pi(C)$ through $\max\pi(C)$.  Since $\pi(C)$ is an interval of integers and the labelled tree coding is recoverable, distinct allowed blocks on this interval give distinct translation classes, and every allowed block occurs.  Hence a fixed core contributes exactly $P_Y(\rho(C)+2)$ translation classes, and $p_{L_Y}(n)=0$ unless $n\equiv1\pmod3$.

\begin{prop}[Raw growth does not force unique ergodicity]\label{prop:raw-not-UE}
Use the notation of Convention~\ref{conv:pullback-complexity}.  If $Y$ has zero entropy, equivalently $P_Y(\ell)=e^{o(\ell)}$, then
\begin{equation}
 \lim_{m\to\infty}p_{L_Y}(3m+1)^{1/(3m+1)}
 =\left(\frac{27}{4}\right)^{1/3}.
 \label{eq:zero-entropy-base}
\end{equation}
For every integer $d\ge2$ there is such a minimal $Y$ for which $L_Y$ carries exactly $d$ ergodic rays and
\[
 p_{L_Y}(3m+1)\asymp p_{L_{\rm Fib}}(3m+1).
\]
Thus even the growth class of $p_L$ up to multiplicative constants does not force a projectively unique carried current.
\end{prop}

\begin{proof}
Summing the core contributions gives
\[
 p_{L_Y}(3m+1)=\sum_{[C]\in\mathcal C_m}P_Y(\rho(C)+2).
\]
Since $0\le\rho(C)\le m-1$, zero entropy gives
\[
 c_m\le p_{L_Y}(3m+1)\le c_m e^{o(m)},
\]
and~\eqref{eq:zero-entropy-base} follows from~\eqref{eq:core-count}.

For the second assertion, Cyr--Kra~\cite[Theorem~1.3]{CK19} construct, for every $d>1$, a minimal finite-alphabet subshift $Z$ with exactly $d$ ergodic probability measures and linear word complexity.  Apply the synchronizing binary recoding of Lemma~\ref{lem:binary-recoding} below.  The resulting minimal binary subshift $Y$ has an affinely homeomorphic invariant-measure simplex and still satisfies $P_Y(\ell)=O(\ell)$, so it has exactly $d$ ergodic measures and zero entropy.  Proposition~\ref{prop:general-measure-cone} transfers these ergodic measures to the ergodic rays carried by $L_Y$.  Since $Y$ is infinite and minimal, Morse--Hedlund gives $P_Y(\ell)\ge \ell+1$; hence $P_Y(\ell)\asymp \ell+1=P_{\rm Fib}(\ell)$, and the displayed core sums are termwise comparable.
\end{proof}

\section{Switch complexity and the dimension of the carried-current cone}\label{sec:switch}

\begin{conv}[Switch-complexity setup]\label{conv:switch}
Convention~\ref{conv:complexity} remains in force throughout this section.  Unless explicitly stated otherwise, $L\subseteq\calH(\partial F_N)$ denotes a nonempty subset lamination carrying a nonzero subset current, so $\SC(L)\ne\{0\}$.  Round graphs, semi-round graphs, their grades, $e$-children, and the associated switch equations are used in the sense of~\cite[Sections~3--4]{Kap17}; we use $n$ for the grade/radius parameter, denoted $r$ in~\cite{Kap17}; the basic weight formalism was recalled in Subsection~\ref{subsec:free-rooted}.  Minimality of $L$ is an additional hypothesis and will always be stated when needed.
\end{conv}

We refine raw complexity by retaining the switch relations among round-graph weights.  Finite-level nullity gives a sufficient condition for finite-dimensionality; the inverse limit gives the exact dimension.

\subsection{Support-restricted round graphs and switch nullity}

Under Convention~\ref{conv:switch}, we use the round and semi-round graphs of~\cite[Sections~3--4]{Kap17}.  For $n\ge1$, let $\mathcal B_n$ be the finite set of $F_N$-translation classes of round graphs of grade $n$ in $X_A$; for $n\ge2$, let $\mathcal J_n$ be the corresponding set of translation classes of semi-round graphs of grade $n$.

\begin{defn}[Round language and switch-nullity complexity]\label{def:round-switch-complexity}
\begin{enumerate}[(1)]
\item The \emph{allowed round language} of $L$ at grade $n$ is
\[
 \mathcal R_n(L):=
 \{[K]\in\mathcal B_n:\operatorname{SCyl}_A(K)\cap L\ne\varnothing\},
\]
and the \emph{round-complexity function} is
\[
 R_L(n):=|\mathcal R_n(L)|.
\]
\item For $n\ge2$, let
\[
 D_n(L):\R^{\mathcal R_n(L)}\longrightarrow\R^{\mathcal J_n}
\]
be the support-restricted switch operator.  Orient each central edge by its label in $A$; the $[J]$-coordinate is the sum of the weights of allowed completions centered at its origin minus the sum for completions centered at its terminus.  Equivalently, $D_n(L)$ is obtained from the switch matrix of~\cite[Definition~4.1]{Kap17} by deleting the columns corresponding to forbidden round graphs.  Put
\[
 W_n(L):=\ker D_n(L),\qquad
 q_L(n):=\dim W_n(L).
\]
The function $q_L$ is the \emph{switch-nullity complexity} of $L$.
\end{enumerate}
\end{defn}

\begin{rem}[Round complexity of the full subset lamination]\label{rem:full-support-round-complexity}
For $L=L_{\mathrm{full}}(F_N)=\calH(\partial F_N)$, put $t_n=(2N-1)^{n-1}$.  A grade-$n$ round graph centered at $1$ is determined by a subset of the sphere of radius $n$ meeting at least two of the $2N$ initial directions.  Each direction contains $t_n$ sphere vertices, so
\[
 R_L(n)=2^{2Nt_n}-2N\,2^{t_n}+(2N-1).
\]
The subtraction removes the empty selection and the selections contained in one initial direction.  Thus $\log_2 R_L(n)\sim2N(2N-1)^{n-1}$, and $R_L$ grows doubly exponentially.
\end{rem}

The two collections of round completions entering $D_n(L)$ are the completion collections from~\cite[Proposition~3.8]{Kap17}; rows which become zero after the support restriction are harmless.

Every $\mu\in\SC(L)$ determines a vector
\[
 \omega_n(\mu):=
 \big(\langle K,\mu\rangle_A\big)_{[K]\in\mathcal R_n(L)}
 \in W_n(L).
\]
For signed elements of $V(L)$ this definition extends linearly.

\begin{thm}[Switch-nullity bound]\label{thm:switch-nullity}
The following hold.
\begin{enumerate}[(1)]
\item One has
\[
 d(L)\le\liminf_{n\to\infty}q_L(n).
\]
\item If there is an integer $M\ge1$ such that $q_L(n)\le M$ for arbitrarily large $n$, then
\[
 d(L)\le M.
\]
\item If $L$ is minimal and $q_L(n)=1$ for arbitrarily large $n$, then $L$ carries a projectively unique nonzero subset current.
\end{enumerate}
\end{thm}

\begin{proof}
For $s\ge n$, the weights of all grade-$s$ round graphs determine, by the splitting formula of~\cite[Corollary~3.5]{Kap17}, the weights of every finite subtree of radius at most $s$, and in particular all grade-$n$ round-graph weights.  Hence on $V(L)$
\[
 \ker\omega_s\subseteq\ker\omega_n
 \qquad(s\ge n).
\]
Moreover,
\[
 \bigcap_{n\ge2}\ker\omega_n=\{0\},
\]
because the round-graph weights of all grades determine all finite-subtree weights, and finite-subtree weights determine a subset current.

Let $E\subseteq V(L)$ be finite-dimensional.  The subspaces
\[
 E\cap\ker\omega_2\supseteq E\cap\ker\omega_3\supseteq\cdots
\]
form a descending chain with zero intersection.  Hence for all sufficiently large $n$,
\[
 E\cap\ker\omega_n=\{0\}.
\]
Thus $\omega_n|_E$ is injective and
\[
 \dim E\le\dim W_n(L)=q_L(n)
\]
for all sufficiently large $n$.  Taking the liminf and then taking the supremum over finite-dimensional subspaces $E\subseteq V(L)$ proves part~(1).  Part~(2) is immediate from part~(1).  Under the hypotheses of part~(3), part~(1) gives $d(L)\le1$; since $\SC(L)\ne\{0\}$, the carried-current cone is a single nonzero ray.
\end{proof}

By rank-nullity,
\begin{equation}
 q_L(n)=|\mathcal R_n(L)|-\operatorname{rank}D_n(L).
 \label{eq:q-rank}
\end{equation}
\subsection{Transition graphs and the round extension excess}

For minimal $L$, an incidence graph inside the switch system forces a large part of $\operatorname{rank}D_n(L)$, yielding a branching analogue of the Rauzy-graph bound.

For $n\ge2$ and $[K]\in\mathcal R_n(L)$, let $c$ be the center of a representative $K$ and put
\[
 K^-:=K\cap\operatorname{Ball}_{X_A}(c,n-1).
\]
Then $K^-$ is a round graph of grade $n-1$, and its translation class depends only on $[K]$.  Thus there is a parent map
\[
 \operatorname{par}_n:\mathcal R_n(L)\longrightarrow\mathcal R_{n-1}(L),
 \qquad
 \operatorname{par}_n([K])=[K^-].
\]
The map $\operatorname{par}_n$ is surjective: an allowed grade-$(n-1)$ round graph occurs as the radius-$(n-1)$ ball about a vertex in a leafless infinite tree representing a point of $L$, and taking the radius-$n$ ball about the same vertex gives an allowed extension.

\begin{defn}[Round extension excess]\label{def:extension-excess}
For $n\ge2$, the \emph{round extension excess} of $L$ at grade $n$ is
\[
 E_L(n):=R_L(n)-R_L(n-1).
\]
By surjectivity of the parent map, $E_L(n)\ge0$.
\end{defn}

If
\[
 e_n(U):=|\operatorname{par}_n^{-1}(U)|
 \qquad(U\in\mathcal R_{n-1}(L)),
\]
then surjectivity of $\operatorname{par}_n$ gives the useful interpretation
\[
 E_L(n)=
 \sum_{U\in\mathcal R_{n-1}(L)}\bigl(e_n(U)-1\bigr).
\]

Let $\mathcal S_n(L)\subseteq\mathcal J_n$ consist of those semi-round classes which occur as an $e$-child, in the sense of~\cite[Definition~3.7]{Kap17}, of some allowed grade-$n$ round graph.  For $[J]\in\mathcal S_n(L)$ choose a representative centered at the midpoint of a positively oriented edge $e$.  By the definition of a semi-round graph, the two truncations
\[
 J\cap\operatorname{Ball}_{X_A}(o(e),n-1),
 \qquad
 J\cap\operatorname{Ball}_{X_A}(t(e),n-1)
\]
are round graphs of grade $n-1$.  Their translation classes depend only on the oriented class of $J$; denote them by
\[
 \partial_0[J],\ \partial_1[J]\in\mathcal R_{n-1}(L).
\]
Define the finite transition multigraph $\mathcal G_n(L)$ to have vertex set $\mathcal R_{n-1}(L)$ and one edge $[J]$ joining $\partial_0[J]$ to $\partial_1[J]$ for each $[J]\in\mathcal S_n(L)$.  Loops and multiple edges are allowed.

\begin{lem}[Connectivity of the transition graph]\label{lem:transition-connected}
If $L$ is minimal, then $\mathcal G_n(L)$ is connected for every $n\ge2$.
\end{lem}

\begin{proof}
Let $U,U'\in\mathcal R_{n-1}(L)$.  Choose representatives centered at $1\in F_N$.  Since $U$ is allowed, there is $S\in L$ such that the rooted tree $T=\Hull_X(S)$ contains $1$ and has grade-$(n-1)$ round pattern $U$ at $1$.  The relative cylinder in $L$ determined by $U'$ is a nonempty open subset of $L$.  Minimality therefore gives $g\in F_N$ such that $gS$ lies in this cylinder.  In particular $1\in\Hull_X(gS)=gT$, so $g^{-1}\in T$, and the grade-$(n-1)$ pattern of $T$ at $g^{-1}$ is $U'$.

Follow the unique edge path in $T$ from $1$ to $g^{-1}$.  The consecutive grade-$(n-1)$ rooted patterns along this path are joined by the grade-$n$ semi-round pattern centered at the midpoint of the intervening edge.  Every such semi-round pattern is allowed.  Hence this path gives a path from $U$ to $U'$ in $\mathcal G_n(L)$.
\end{proof}

\begin{lem}[Weighted incidence identity]\label{lem:weighted-incidence}
Fix $n\ge2$.  Suppose that $L$ is minimal, let $0\ne\mu\in\SC(L)$, and for $[K]\in\mathcal R_n(L)$ put
\[
 x_K:=\langle K,\mu\rangle_A.
\]
Define the weighted parent-summing operator
\[
 \mathsf{Par}_n^\mu:\R^{\mathcal R_n(L)}\longrightarrow\R^{\mathcal R_{n-1}(L)},
 \qquad
 (\mathsf{Par}_n^\mu z)_U=
 \sum_{\operatorname{par}_n([K])=U}x_Kz_K.
\]
For an allowed semi-round class $[J]\in\mathcal S_n(L)$, give its central edge $e$ the preceding orientation.  Then
\begin{equation}
 \mathsf{Par}_n^\mu D_n(L)^t\mathbf e_{[J]}
 =\langle J,\mu\rangle_A
 \bigl(\mathbf e_{\partial_0[J]}-\mathbf e_{\partial_1[J]}\bigr).
 \label{eq:weighted-incidence}
\end{equation}
In particular, if $\supp(\mu)=L$, then the scalar $\langle J,\mu\rangle_A$ is positive.
\end{lem}

\begin{proof}
Let $\mathcal E_0(J)$ and $\mathcal E_1(J)$ denote the two finite multisets of allowed grade-$n$ round completions of $J$, centered respectively at $o(e)$ and $t(e)$.  The switch equation of~\cite[Proposition~3.8]{Kap17} uses exactly these multiplicities.  Every $K\in\mathcal E_0(J)$ has parent $\partial_0[J]$, while every $K\in\mathcal E_1(J)$ has parent $\partial_1[J]$.  Consequently
\[
 \mathsf{Par}_n^\mu D_n(L)^t\mathbf e_{[J]}
 =\left(\sum_{K\in\mathcal E_0(J)}x_K\right)\mathbf e_{\partial_0[J]}
 -\left(\sum_{K\in\mathcal E_1(J)}x_K\right)\mathbf e_{\partial_1[J]},
\]
with multiplicities understood; if the parent classes coincide, the contributions add in the same coordinate.

The splitting formula for the semi-round graph $J$ says that each of the two unsigned sums equals $\langle J,\mu\rangle_A$.  The resulting equality proves~\eqref{eq:weighted-incidence}.  If $\supp(\mu)=L$ and $J$ is allowed, then some allowed completion of $J$ has positive $\mu$-weight, and hence the common value $\langle J,\mu\rangle_A$ is positive.
\end{proof}

\begin{prop}[Incidence rank inside the switch system]\label{prop:switch-rank}
Suppose that $L$ is minimal.  Then for every $n\ge2$,
\begin{equation}
 \operatorname{rank}D_n(L)\ge R_L(n-1)-1.
 \label{eq:switch-rank-lower}
\end{equation}
Consequently,
\begin{equation}
 q_L(n)\le E_L(n)+1.
 \label{eq:q-extension}
\end{equation}
\end{prop}

\begin{proof}
Choose $0\ne\mu\in\SC(L)$.  Minimality gives $\supp(\mu)=L$, so every allowed round and semi-round pattern has positive $\mu$-weight.
Use the operator $\mathsf{Par}_n^\mu$ from Lemma~\ref{lem:weighted-incidence}.  That lemma shows that, after discarding loop columns, the matrix of $\mathsf{Par}_n^\mu D_n(L)^t$ on the allowed semi-round coordinates is an incidence matrix of $\mathcal G_n(L)$ with each column multiplied by the positive scalar $\langle J,\mu\rangle_A$.  By Lemma~\ref{lem:transition-connected}, this graph is connected, so its incidence matrix has rank $R_L(n-1)-1$.  Hence
\[
 \operatorname{rank}D_n(L)
 =\operatorname{rank}D_n(L)^t
 \ge \operatorname{rank}\bigl(\mathsf{Par}_n^\mu D_n(L)^t\bigr)
 =R_L(n-1)-1,
\]
which proves~\eqref{eq:switch-rank-lower}.  Combining~\eqref{eq:switch-rank-lower} with~\eqref{eq:q-rank} gives~\eqref{eq:q-extension}.
\end{proof}

\begin{thm}[Round extension-excess bound]\label{thm:extension-excess}
Suppose that $L$ is minimal.  Then
\begin{equation}
 \boxed{
 d(L)\le 1+\liminf_{n\to\infty}E_L(n)
 =1+\liminf_{n\to\infty}\bigl(R_L(n)-R_L(n-1)\bigr).
 }
 \label{eq:extension-dim-bound}
\end{equation}
In particular,
\begin{equation}
 d(L)\le 1+\liminf_{n\to\infty}\frac{R_L(n)}n.
 \label{eq:linear-dim-bound}
\end{equation}
If the latter liminf is finite, the right-hand side may be replaced by
\begin{equation}
 1+\left\lfloor\liminf_{n\to\infty}\frac{R_L(n)}n\right\rfloor.
 \label{eq:linear-dim-floor}
\end{equation}
Consequently $R_L(n)=O(n)$ implies $d(L)<\infty$.
\end{thm}

\begin{proof}
The first inequality follows from Theorem~\ref{thm:switch-nullity} and Proposition~\ref{prop:switch-rank}.  Since
\[
 R_L(n)=R_L(1)+\sum_{j=2}^n E_L(j)
\]
and $E_L(j)\ge0$, the liminf of the terms $E_L(j)$ is at most the liminf of their Ces\`aro averages.  Therefore
\[
 \liminf_n E_L(n)\le \liminf_n\frac{R_L(n)}n,
\]
which proves~\eqref{eq:linear-dim-bound}.  Since $d(L)$ and $E_L(n)$ are integer-valued,~\eqref{eq:linear-dim-floor} follows as well.
\end{proof}

\begin{cor}[Necessary complexity for many currents]\label{cor:necessary-complexity}
Under the hypotheses of Theorem~\ref{thm:extension-excess}, the following hold.
\begin{enumerate}[(1)]
\item If $d(L)=d<\infty$, then
\[
 E_L(n)\ge d-1
\]
for all sufficiently large $n$, and hence
\begin{equation}
 R_L(n)\ge (d-1)n-O(1).
 \label{eq:necessary-linear}
\end{equation}
\item If $d(L)=\infty$, then
\begin{equation}
 E_L(n)\longrightarrow\infty
 \qquad\text{and}\qquad
 \frac{R_L(n)}n\longrightarrow\infty.
 \label{eq:infinite-superlinear}
\end{equation}
\end{enumerate}
\end{cor}

\begin{proof}
If $d(L)=d<\infty$,~\eqref{eq:extension-dim-bound} gives $\liminf E_L(n)\ge d-1$.  Since $E_L(n)$ is integer-valued, $E_L(n)\ge d-1$ for all sufficiently large $n$, and summing the increments proves~\eqref{eq:necessary-linear}.  If $d(L)=\infty$, the same bound forces $\liminf E_L(n)=\infty$, which is exactly $E_L(n)\to\infty$.  The Ces\`aro averages of a sequence tending to infinity also tend to infinity, giving~\eqref{eq:infinite-superlinear}.
\end{proof}

\subsection{Restriction maps and stable switch sets}\label{subsec:stable-switch}

Finite-level switch solutions need not extend coherently to all higher grades, so finite-level nullity need not equal the true dimension.  We isolate the extendable part.

\begin{defn}[Finite switch cone and polytope]\label{def:finite-switch-polytope}
For $n\ge2$ define the \emph{finite switch cone}
\[
 Q_n(L):=W_n(L)\cap\R_{\ge0}^{\mathcal R_n(L)},
\]
and, for $x\in Q_n(L)$, put
\[
 \|x\|_1:=\sum_{[K]\in\mathcal R_n(L)}x_K.
\]
The corresponding \emph{normalized finite switch polytope} is
\[
 \Delta_n(L):=\{x\in Q_n(L):\|x\|_1=1\}.
\]
Because $L$ carries a nonzero current, $\Delta_n(L)$ is nonempty.
\end{defn}

For $s\ge n$, Corollary~3.5 of~\cite{Kap17} gives a linear restriction map
\[
 \operatorname{Res}_{s,n}:\R^{\mathcal R_s(L)}\to\R^{\mathcal R_n(L)}
\]
by summing over all grade-$s$ round extensions of a grade-$n$ round graph with the same center.

\begin{lem}\label{lem:restriction-switch}
For $s\ge n\ge2$,
\[
 \operatorname{Res}_{s,n}(Q_s(L))\subseteq Q_n(L),\qquad
 \operatorname{Res}_{s,n}(\Delta_s(L))\subseteq\Delta_n(L),
\]
and
\[
 \operatorname{Res}_{t,n}=\operatorname{Res}_{s,n}\operatorname{Res}_{t,s}\qquad(t\ge s\ge n).
\]
\end{lem}

\begin{proof}
The composition rule and preservation of the $\ell^1$-normalization follow directly by iterating the disjoint cylinder splittings.  We verify the switch relations directly.

Fix a grade-$n$ semi-round graph $J$.  Each completion of $J$ to a grade-$n$ round graph admits a finite disjoint family of grade-$s$ round refinements.  Equivalently, if one first refines $J$ to all compatible grade-$s$ semi-round graphs and then completes on either side of the central edge, one obtains exactly the same grade-$s$ round patterns, with the same multiplicities, as by first completing at grade $n$ and then refining.  Therefore, for every $x\in\R^{\mathcal R_s(L)}$,
\[
 \bigl(D_n(L)\operatorname{Res}_{s,n}x\bigr)_J
 =\sum_{J'}\bigl(D_s(L)x\bigr)_{J'},
\]
where the finite sum ranges over the allowed grade-$s$ semi-round refinements $J'$ of $J$, with the multiplicities coming from the cylinder splittings of~\cite[Proposition~3.8 and Corollary~3.5]{Kap17}.  Hence $D_s(L)x=0$ implies $D_n(L)\operatorname{Res}_{s,n}x=0$.

If a grade-$n$ round pattern is forbidden by $L$, then every grade-$s$ extension of it is forbidden as well.  Thus $\operatorname{Res}_{s,n}$ also preserves the support restriction and nonnegativity.  Therefore
\[
 \operatorname{Res}_{s,n}(Q_s(L))\subseteq Q_n(L).
\]
Together with preservation of the $\ell^1$-normalization, it gives the corresponding assertion for $\Delta_s(L)$.
\end{proof}

\begin{defn}[Stable switch set and stable finite-level dimension]\label{def:stable-switch}
Let $n\ge2$.
\begin{enumerate}[(1)]
\item The \emph{stable switch set} at grade $n$ is
\[
 \Delta_n^\infty(L):=
 \bigcap_{s\ge n}\operatorname{Res}_{s,n}(\Delta_s(L)).
\]
The sets in the intersection are nested nonempty compact convex subsets of $\Delta_n(L)$; hence $\Delta_n^\infty(L)$ is nonempty.
\item The \emph{stable finite-level dimension} is
\[
 s_L(n):=1+\operatorname{affdim}\Delta_n^\infty(L).
\]
\end{enumerate}
\end{defn}

By Definition~\ref{def:normalized-current-base}, $\mathcal M_A(L)$ denotes the normalized carried-current base.  Equivalently, for $\mu\in\SC(L)$, one has $\mu\in\mathcal M_A(L)$ exactly when
\[
 \sum_{[K]\in\mathcal B_n}\langle K,\mu\rangle_A=1
\]
for one, and hence every, grade $n$.

\begin{thm}[Stable switch characterization]\label{thm:stable-switch}
The following hold.
\begin{enumerate}[(1)]
\item For every $n\ge2$, the image of $\mathcal M_A(L)$ under the grade-$n$ weight map is exactly
\begin{equation}
 \Delta_n^\infty(L).
 \label{eq:image-stable}
\end{equation}
\item One has
\begin{equation}
 \boxed{
 d(L)=\sup_{n\ge2}s_L(n)
 =1+\sup_{n\ge2}\operatorname{affdim}\Delta_n^\infty(L).
 }
 \label{eq:exact-dimension}
\end{equation}
Moreover, the sequence $s_L(n)$ is nondecreasing and
\begin{equation}
 s_L(n)\le q_L(n)
 \label{eq:s-less-q}
\end{equation}
for every $n$.
\item If $L$ is minimal, then $L$ carries a projectively unique nonzero subset current if and only if every $\Delta_n^\infty(L)$ is a single point.
\end{enumerate}
\end{thm}

\begin{proof}
Every normalized current carried by $L$ supplies compatible vectors in $\Delta_s(L)$ at all grades.  Hence its grade-$n$ vector belongs to every $\operatorname{Res}_{s,n}(\Delta_s(L))$, proving one inclusion in~\eqref{eq:image-stable}.

Conversely, let $x_n\in\Delta_n^\infty(L)$.  For every $m\ge n$, choose $x_m^{(m)}\in\Delta_m(L)$ with
\[
 \operatorname{Res}_{m,n}(x_m^{(m)})=x_n.
\]
Restricting $x_m^{(m)}$ to the intermediate grades gives a compatible finite string of vectors from grade $n$ through grade $m$.  Consider the product of the compact polytopes $\prod_{j\ge n}\Delta_j(L)$.  For $j\ge n$ let
\[
 C_j:=\{(x_i)_{i\ge n}:\operatorname{Res}_{j+1,j}(x_{j+1})=x_j\},
\]
and also impose the closed condition that the $n$-th coordinate equals the prescribed vector $x_n$.  Every finite collection of these conditions is satisfiable by taking one of the strings coming from $x_m^{(m)}$ with $m$ sufficiently large and choosing the remaining coordinates arbitrarily.  The finite-intersection property and compactness therefore give a fully compatible sequence
\[
 (x_j)_{j\ge n},\qquad
 \operatorname{Res}_{j+1,j}(x_{j+1})=x_j.
\]
Extend it to the finitely many smaller grades by restriction.

Extend each $x_j$ by zero on forbidden round classes.  The grade-$j$ round cylinders centered at the identity form a finite clopen partition of $\calT_1(X_A)$, and these partitions refine one another.  Compatibility under the maps $\operatorname{Res}_{j+1,j}$ says exactly that the values prescribed by $x_j$ are finitely additive and consistent under refinement.  Thus, for any clopen set which is a union of sufficiently fine round cylinders, its mass is defined by summing the corresponding coordinates of $x_j$, independently of the sufficiently large choice of $j$.  These clopen sets form an algebra generating the topology.  Moreover, a decreasing sequence of clopen sets with empty intersection is eventually empty by compactness, so this finitely additive probability is a premeasure on the clopen algebra.  It therefore extends uniquely to a Borel probability measure $\nu$ on $\calT_1(X_A)$.

We next verify root-change invariance.  Fix $a\in A$.  For each $j\ge2$, the rooted cylinders $C_J$ of grade-$j$ semi-round graphs centered at the midpoint of $[1,a]$ form a finite clopen partition of $\calT_{1,a}(X_A)$.  These partitions refine and generate its topology, since each semi-round pattern determines the radius-$(j-1)$ ball about $1$.  The cylinders of round completions centered at $1$ partition $C_J$; those centered at $a$, translated by $a^{-1}$, partition $r_aC_J$.  Their masses are the two full sums in the switch equation, hence are equal.  Every clopen subset of the domain is a union of sufficiently fine such cylinders, so
\[
 \nu(r_aC)=\nu(C)
\]
for every clopen $C\subseteq\calT_{1,a}(X_A)$.  Uniqueness of finite Borel measures extends the equality to all Borel subsets of the domain.  The inverse maps give invariance for the letters in $A^{-1}$ as well.

Now let $g=s_1\cdots s_\ell$ be a reduced word in the free basis alphabet and suppose $g\in T$.  Since $T$ is a subtree containing $1$, every prefix $s_1\cdots s_i$ lies in $T$.  Thus the partial root change $r_g(T)=g^{-1}T$ is the composition of the successive one-edge root changes along this path.  Hence $\nu$ is invariant under every partial root-change map.  Theorem~\ref{thm:KN-rooted} therefore gives a normalized subset current $\mu$.

Only patterns whose cylinders meet $L$ have nonzero coordinates.  Let
\[
 K_L:=\{\Hull_X(S):S\in L,\ 1\in\Hull_X(S)\}\subseteq\calT_1(X_A).
\]
By the rooted boundary--hull homeomorphism in Subsection~\ref{subsec:free-rooted}, the set $K_L$ is closed in $\calT_1(X_A)$.  If a rooted tree does not belong to $K_L$, then, because rooted cylinders form a clopen basis, some finite rooted ball occurring in it is forbidden by $L$; the corresponding round cylinder has $\nu$-measure zero.  Hence $\supp(\nu)\subseteq K_L$.  Since $L$ is $F_N$-invariant, its rooted transversal satisfies
\[
 \Sat(K_L)=L.
\]
Lemma~\ref{lem:support-transfer} now gives $\supp(\mu)\subseteq L$, establishing the reverse inclusion in~\eqref{eq:image-stable}.

The grade maps separate currents.  Therefore the affine dimension of $\mathcal M_A(L)$ is the supremum of the affine dimensions of its finite-grade images.  For a direct verification, any finite affinely independent collection of normalized currents remains affinely independent after projection to all sufficiently large grades, by the same descending-kernel argument as in the proof of Theorem~\ref{thm:switch-nullity}.  Hence
\[
 \operatorname{affdim}\mathcal M_A(L)
 =\sup_n\operatorname{affdim}\Delta_n^\infty(L).
\]
Since $d(L)=1+\operatorname{affdim}\mathcal M_A(L)$, the preceding equality proves~\eqref{eq:exact-dimension}.

For $s\ge n$, part~(1) and $\omega_n=\operatorname{Res}_{s,n}\circ\omega_s$ give
\[
 \operatorname{Res}_{s,n}(\Delta_s^\infty(L))=\Delta_n^\infty(L).
\]
Thus the affine dimensions, and hence $s_L(n)$, are nondecreasing.  Finally,~\eqref{eq:s-less-q} follows because $\Delta_n^\infty(L)$ lies in the affine hyperplane $\sum_K x_K=1$ inside $W_n(L)$.  The last assertion follows from part~(1), since the grade maps separate currents.
\end{proof}

Since $s_L(n)$ is a nondecreasing integer sequence and $d(L)=\sup_n s_L(n)$, finite dimensionality is equivalent to eventual stabilization of the stable finite-level dimensions $s_L(n)$.

Thus
\[
 \boxed{
 d(L)=\lim_{n\to\infty}s_L(n)
 \le\liminf_{n\to\infty}q_L(n)
 \le 1+\liminf_{n\to\infty}E_L(n)
 \le 1+\liminf_{n\to\infty}\frac{R_L(n)}n
 }
\]
for every minimal $L$ carrying a nonzero subset current.  The first term is exact, the second retains all finite-level switch directions, the third records only one-step round extension excess, and the last is the coarsest growth-rate bound.

\begin{cor}[Finite-dimensional cones have finitely many ergodic rays]\label{cor:finite-dim-ergodic}
Suppose that $L$ is minimal and that
\[
 d(L)=d<\infty.
\]
Then $L$ carries exactly $d$ ergodic rays.  More precisely, the normalized current space $\mathcal M_A(L)$ is an ordinary $(d-1)$-simplex, and after choosing a nonzero representative on each of its $d$ extreme rays, every carried current has a unique expression as a nonnegative linear combination of these representatives.  Consequently, if for some integer $M\ge1$ one has
\[
 q_L(n)\le M
\]
for arbitrarily large $n$, then $L$ carries at most $M$ ergodic rays.
\end{cor}

\begin{proof}
Via Theorem~\ref{thm:KN-rooted}, $\mathcal M_A(L)$ is the simplex of invariant probability measures for the countable root-change equivalence relation on the rooted-tree transversal associated to $L$.  The standard ergodic decomposition theorem for countable Borel equivalence relations identifies its extreme points with the ergodic invariant probabilities and makes $\mathcal M_A(L)$ a Choquet simplex; compare~\cite[Theorem~3.3]{KM04}.  A finite-dimensional Choquet simplex is an ordinary finite simplex.  Since
\[
 \operatorname{affdim}\mathcal M_A(L)=d(L)-1=d-1,
\]
it has exactly $d$ extreme points.  The last assertion follows from Theorem~\ref{thm:switch-nullity}.
\end{proof}

\begin{proof}[Proof of Theorem~\ref{mainthm:complexity}]
The upper bounds and the implication $R_L(n)=O(n)\Rightarrow d(L)<\infty$ follow from Theorem~\ref{thm:extension-excess}.  Corollary~\ref{cor:necessary-complexity} gives the asserted lower bounds when $d(L)<\infty$ and the divergence conclusions when $d(L)=\infty$.  Finally, Corollary~\ref{cor:finite-dim-ergodic} shows that, when $d(L)=d<\infty$, exactly $d$ ergodic rays are carried by $L$.
\end{proof}

\subsection{Reduction to ordinary Rauzy graphs in the pullback examples}

For the supports $L_Y$ coming from~\eqref{eq:p-to-Z}, the support-restricted switch system can be identified explicitly with the ordinary Rauzy system of $Y$.

\begin{prop}[Rauzy reduction]\label{prop:rauzy-reduction}
Let $Y\subseteq\{0,1\}^{\Z}$ be an infinite minimal subshift and let $L_Y$ be the associated branching subset lamination of Convention~\ref{conv:pullback-complexity}.  Then:
\begin{enumerate}[(1)]
\item For every $n\ge2$,
\[
 |\mathcal R_n(L_Y)|=P_Y(2n).
\]
\item For every $n\ge2$,
\[
 q_{L_Y}(n)=P_Y(2n)-P_Y(2n-1)+1.
\]
\item The normalized carried-current base $\mathcal M_A(L_Y)$ is affinely homeomorphic to the simplex of shift-invariant probability measures on $Y$.
\end{enumerate}
\end{prop}

\begin{proof}
A radius-$n$ rooted ball in $T_x$ records precisely the coordinates
\[
 y_{-n},y_{-n+1},\ldots,y_{n-1}.
\]
Indeed every vertex of $X_B$ at distance at most $n-1$ from the root is saturated in the round ball, and the negative $u$-edge at the lowest $\pi$-height records the additional value $y_{-n}$; no value outside the displayed interval is visible.  The coding is recoverable from the labelled tree, so grade-$n$ round patterns are in bijection with the allowed words of length $2n$.  The bijection proves part~(1).

For the switch equations, root change across a $v$-edge does not change the symbolic point because $v\in\ker\pi$; the two completion multisets are identical, so the corresponding switch rows are tautological.  Across a $u$-edge, the two radius-$n$ rooted balls overlap in the data of an allowed word of length $2n-1$.  More explicitly, the allowed $u$-centered semi-round classes are in bijection with the allowed $(2n-1)$-words, while an allowed $2n$-word gives the corresponding round class and joins its length-$(2n-1)$ prefix to its length-$(2n-1)$ suffix.  The completion multiplicities in the switch equation are exactly the usual incidence multiplicities.  Thus, after deleting the tautological rows, $D_n(L_Y)$ is the incidence matrix of the ordinary Rauzy graph whose vertices are the allowed $(2n-1)$-words and whose directed edges are the allowed $2n$-words.  Minimality makes this Rauzy graph strongly connected.  Hence its incidence matrix has rank
\[
 P_Y(2n-1)-1,
\]
and rank-nullity gives part~(2).

Part~(3) follows from Proposition~\ref{prop:general-measure-cone}: the $B$-action on $\Omega_Y$ factors through the shift action of $\Z$ on $Y$, and the pullback identifies the two invariant-measure simplices.
\end{proof}

For Fibonacci, $P_{\rm Fib}(\ell)=\ell+1$, so Proposition~\ref{prop:rauzy-reduction} gives
\[
 q_{L_{\rm Fib}}(n)=2
 \qquad(n\ge2),
\]
whereas Theorem~\ref{thm:symbolic} gives $d(L_{\rm Fib})=1$.  Equivalently,
\[
 s_{L_{\rm Fib}}(n)=1
\]
for every $n$.  Thus even in the simplest uniquely ergodic example the finite-level switch cone has an extra direction that disappears in the inverse limit.  This discrepancy is the exact branching analogue of the familiar fact that finite Rauzy circulation spaces can have dimension greater than the invariant-measure simplex of the subshift.

\begin{rem}[Unique ergodicity with exponential complexity]\label{rem:UE-exponential}
Strictly ergodic finite-alphabet subshifts with positive topological entropy are classical~\cite{Gri73}.  The synchronizing binary recoding of Lemma~\ref{lem:binary-recoding} preserves unique ergodicity and positive entropy.  For the resulting binary subshift $Y$, Propositions~\ref{prop:general-measure-cone} and~\ref{prop:rauzy-reduction} give a minimal genuinely branching subset lamination with $d(L_Y)=1$ and
\[
 \lim_{n\to\infty}\frac1n\log R_{L_Y}(n)=2h_{\mathrm{top}}(Y)>0.
\]
Thus projective unique ergodicity does not force subexponential round complexity.
\end{rem}

\subsection{Finite-dimensionality from symbolic low complexity}

The measure-cone correspondence transfers classical complexity results to the branching setting.

\begin{cor}[Boshernitzan bound for the branching pullback family]\label{cor:branching-Boshernitzan}
Let $Y\subseteq\{0,1\}^{\Z}$ be an infinite minimal subshift, with $P_Y$ and $L_Y$ as in Convention~\ref{conv:pullback-complexity}.  Then:
\begin{enumerate}[(1)]
\item If, for an integer $k\ge2$,
\begin{equation}
 \liminf_{\ell\to\infty}\frac{P_Y(\ell)}\ell<k,
 \label{eq:Bos-condition}
\end{equation}
then the associated minimal genuinely branching subset lamination $L_Y\subseteq\calH(\partial F_4)$ satisfies
\begin{equation}
 d(L_Y)\le k-1.
 \label{eq:Bos-dimension}
\end{equation}
In particular, $L_Y$ carries at most $k-1$ ergodic rays.
\item If
\[
 \limsup_{\ell\to\infty}\frac{P_Y(\ell)}\ell<3,
\]
then $L_Y$ carries a projectively unique nonzero subset current.
\end{enumerate}
\end{cor}

\begin{proof}
For part~(1), use Boshernitzan's theorem in the form of~\cite[Theorem~1.1]{CK19} (see~\cite{Bos84} for the original minimal-subshift result).  Since $Y$ is infinite and minimal, every ergodic invariant probability is nonatomic and has generic points; thus~\eqref{eq:Bos-condition} implies that $Y$ has at most $k-1$ ergodic invariant probability measures.  The invariant-measure set is a Choquet simplex, so if it has $e<\infty$ extreme points, the cone of finite invariant measures has linear dimension $e$.  Proposition~\ref{prop:general-measure-cone} therefore gives~\eqref{eq:Bos-dimension}.

For part~(2), suppose $\limsup P_Y(\ell)/\ell<3$ and $Y$ is not uniquely ergodic.  Choose an ergodic measure and a generic point for it.  Minimality makes the orbit closure of that point equal to $Y$, so~\cite[Theorem~1.2]{CK19} with $k=3$ allows at most one nonatomic generic measure.  The conclusion contradicts the existence of at least two ergodic invariant probabilities, each nonatomic and generic in the preceding sense.
\end{proof}

We now show that the linear threshold is sharp: below any prescribed superlinear scale one can realize arbitrary metrizable Choquet simplices.

We first record a standard recoding device, included because the tree coding of Section~3 is binary whereas the low-complexity realization theorem of Cecchi Bernales--Donoso is stated for finite-alphabet subshifts.

\begin{lem}[Synchronizing binary recoding]\label{lem:binary-recoding}
Let $\mathcal A$ be a finite alphabet and let $Z\subseteq\mathcal A^{\Z}$ be a nonempty minimal subshift.  Let $P_Z(\ell)$ denote its word-complexity function.  Then there exist an integer $m\ge4$ and a minimal binary subshift $Y\subseteq\{0,1\}^{\Z}$, with word-complexity function $P_Y(\ell)$, such that:
\begin{enumerate}[(1)]
\item the invariant probability simplices of $Z$ and $Y$ are affinely homeomorphic;
\item if $Z$ is Toeplitz, then $Y$ is Toeplitz;
\item for every $\ell\ge1$,
\begin{equation}
 P_Y(\ell)\le m\,P_Z\!\left(\left\lceil\frac \ell m\right\rceil+2\right).
 \label{eq:binary-complexity-bound}
\end{equation}
\end{enumerate}
\end{lem}

\begin{proof}
Choose $m$ sufficiently large and an injective constant-length code $c:\mathcal A\to\{0,1\}^m$ such that every codeword begins with $11$, ends with $0$, and has no other occurrence of $11$.  For $x\in Z$, concatenate the words $c(x_k)$ and put
\[
 Y=\bigcup_{j=0}^{m-1}\sigma^j c(Z).
\]
The marker $11$ identifies block beginnings, so the $m$ phases are pairwise disjoint and clopen and $c:Z\to c(Z)$ conjugates $\sigma$ with $\sigma^m$ on the aligned phase.  Minimality follows.  Averaging a $Z$-invariant probability measure over the $m$ phases gives a $Y$-invariant probability measure; conversely, every $Y$-invariant probability assigns mass $1/m$ to each phase, so the inverse is the normalized restriction
\[
 \nu\longmapsto (c^{-1})_*\bigl(m\,\nu|_{c(Z)}\bigr).
\]
Thus the invariant-measure simplices are affinely homeomorphic.  Constant-length coding also preserves the Toeplitz property.  Finally, a length-$\ell$ word of $Y$ is determined by its phase and by at most $\lceil \ell/m\rceil+2$ consecutive symbols of $Z$, which gives~\eqref{eq:binary-complexity-bound}.
\end{proof}

\begin{thm}[Choquet universality below every superlinear scale]\label{thm:slow-universality}
Let $\mathcal K$ be a nonempty metrizable Choquet simplex and let $(a_n)$ be any sequence of positive real numbers such that
\begin{equation}
 \frac{n}{a_n}\longrightarrow0.
 \label{eq:superlinear-scale}
\end{equation}
Then there exists an uncountable minimal genuinely branching subset lamination
\[
 L\subseteq\calH(\partial F_4)
\]
constructed with the standard basis $A=\{a_0,a_1,b_0,b_1\}$ of $F_4$, such that
\begin{equation}
 \mathcal M_A(L)\cong_{\rm aff}\mathcal K
 \label{eq:slow-univ-simplex}
\end{equation}
and
\begin{equation}
 R_L(n)=o(a_n).
 \label{eq:slow-round}
\end{equation}
In particular, if $\mathcal K$ has infinite-dimensional affine hull, then $d(L)=\infty$ while the round complexity can be made smaller than any prescribed superlinear scale.
\end{thm}

\begin{proof}
If $\mathcal K$ is a singleton, take the Fibonacci subshift.  Then $R_L(n)=2n+1$ by Proposition~\ref{prop:rauzy-reduction}, so~\eqref{eq:superlinear-scale} gives~\eqref{eq:slow-round}.

Assume now that $\mathcal K$ is not a singleton.  Define
\[
 b_\ell:=\inf_{j\ge \ell}a_j.
\]
Since $a_j/j\to\infty$, we have $b_\ell/\ell\to\infty$.  Cecchi Bernales--Donoso~\cite[Theorem~1.4]{CD26} prove that there exists a minimal Toeplitz subshift $Z$ over a finite alphabet whose invariant-measure simplex is affinely homeomorphic to $\mathcal K$ and whose word complexity satisfies
\begin{equation}
 P_Z(\ell)=o(b_\ell).
 \label{eq:CD-low-complexity}
\end{equation}
Apply Lemma~\ref{lem:binary-recoding} and let $m$ be its code length.  The resulting binary Toeplitz subshift $Y$ has the same invariant-measure simplex.  Put
\[
 k_n:=\left\lceil\frac{2n}{m}\right\rceil+2.
\]
For all sufficiently large $n$, $k_n\le n$, and therefore $b_{k_n}\le a_n$.  By~\eqref{eq:binary-complexity-bound} and~\eqref{eq:CD-low-complexity},
\[
 \frac{P_Y(2n)}{a_n}
 \le m\frac{P_Z(k_n)}{b_{k_n}}\longrightarrow0.
\]
Pull $Y$ back to $F(u,v)$ by the construction used above and perform the $F_4$ tree coding.  Proposition~\ref{prop:general-measure-cone} gives~\eqref{eq:slow-univ-simplex}, while Proposition~\ref{prop:rauzy-reduction} gives
\[
 R_L(n)=P_Y(2n)=o(a_n).
\]
The system $Z$, hence $Y$, is infinite because its invariant-measure simplex is not a singleton; therefore $L$ is uncountable and genuinely branching by the tree construction.
\end{proof}

Together with Theorem~\ref{mainthm:complexity}, this proves that linear round complexity is the sharp universal upper-growth threshold for finite-dimensionality: below every prescribed superlinear scale, arbitrary metrizable Choquet simplices can still occur.  This concerns what an upper complexity bound forces about $d(L)$, not a converse restriction on complexity in terms of $d(L)$; compare Remark~\ref{rem:UE-exponential}.

\begin{rem}[Finite-alphabet Toeplitz inputs]\label{rem:nonabelian-many-measures}
For every integer $m>1$, Cecchi Bernales--Cortez--G\'omez~\cite[Theorem~1.3]{CCG24} construct a minimal Toeplitz $F_2$-subshift over an $m$-letter alphabet with at least $m$ ergodic invariant probabilities.  The coding of Section~3 extends to this alphabet by using $F(a_1,\ldots,a_m,b_1,\ldots,b_m)$ as the ambient group; the encoded trees remain exactly $4$-regular, and the same recoverability and root-change arguments apply.  Thus these systems yield minimal genuinely branching subset laminations with at least $m$ ergodic rays, which can then be transported by malnormal induction.  The cited theorem alone does not give exactly $m$ rays.
\end{rem}

\section{Change of marking and round-complexity growth}\label{sec:marking-invariance}

Round complexity was defined using a free basis.  Its coarse growth type is nevertheless intrinsic: changing the basis, or more generally the finite marked graph, changes radii only linearly and pattern counts by a bounded factor.  This section requires neither minimality nor a carried current.

\subsection{Marked graphs and coarse growth type}

\begin{conv}[Marked-graph setup]\label{conv:marked-round}
Let $F_N$ be a free group of rank $N\ge2$.  A \emph{marked graph} $\Gamma$ means a finite connected core graph equipped with a marking identifying $\pi_1(\Gamma)$ with $F_N$; see~\cite{CV86}.  Give every edge of $\Gamma$ unit length and let
\[
 X_\Gamma=\widetilde\Gamma
\]
be the universal covering tree, with the induced free cocompact $F_N$-action and the marking identification $\partial X_\Gamma\cong\partial F_N$.  For $S\in\calH(\partial F_N)$ put
\[
 T_\Gamma(S):=\Hull_{X_\Gamma}(S).
\]
\end{conv}

\begin{defn}[Marked-graph round complexity]\label{def:marked-round-complexity}
If $L\subseteq\calH(\partial F_N)$ is a nonempty subset lamination and $n\ge1$, let $\mathcal R_n^\Gamma(L)$ be the set of $F_N$-translation classes of pointed finite trees
\[
 \bigl(T_\Gamma(S)\cap \operatorname{Ball}_{X_\Gamma}(x,n),x\bigr),
 \qquad S\in L,\quad x\in V T_\Gamma(S).
\]
The \emph{marked-graph round-complexity function} is
\[
 R_L^\Gamma(n):=|\mathcal R_n^\Gamma(L)|.
\]
When $\Gamma$ is the rose associated to the basis $A$ in Convention~\ref{conv:complexity}, the marked-graph function is exactly the round complexity $R_L(n)=|\mathcal R_n(L)|$ of Definition~\ref{def:round-switch-complexity}.
\end{defn}

Restriction from radius $n+1$ to $n$ is surjective on occurring patterns, so $R_L^\Gamma(n)$ is nondecreasing.

\begin{defn}[Linear-rescaling growth type]\label{def:linear-growth-type}
Let $f,g:\N\to\R_{\ge0}$ be nondecreasing.  We write
\[
 f\preccurlyeq_{\rm lin} g
\]
if there are constants $a\ge1$, $b\ge0$, and $M\ge1$ such that
\[
 f(n)\le M g(\lceil an+b\rceil)
\]
for every $n\ge1$.  We write $f\simeq_{\rm lin}g$ if both $f\preccurlyeq_{\rm lin}g$ and $g\preccurlyeq_{\rm lin}f$, and call the resulting equivalence class the \emph{linear-rescaling growth type} of $f$.
\end{defn}

It is immediate from monotonicity that $\simeq_{\rm lin}$ is an equivalence relation.

\subsection{Bounded cancellation and finite locality}

Let $\Gamma,\Gamma'$ be two marked graphs for $F_N$.  Choose a marking-preserving piecewise-linear cellular homotopy equivalence
\[
 f:\Gamma\longrightarrow\Gamma'
\]
and an $F_N$-equivariant lift
\[
 \widetilde f:X_\Gamma\longrightarrow X_{\Gamma'}.
\]
The lift $\widetilde f$ is a quasi-isometry, so for some $\lambda\ge1$ and $\kappa\ge0$,
\begin{equation}
 d_{X_{\Gamma'}}(\widetilde f(p),\widetilde f(q))
 \ge \lambda^{-1}d_{X_\Gamma}(p,q)-\kappa.
 \label{eq:qi-lower}
\end{equation}
By marked-graph bounded cancellation~\cite[Lemma~2.3.1]{BFH00}; see also~\cite{Coo87}, there is $C\ge0$ such that rays agreeing from $x$ through $v$ have tightened $\widetilde f$-images agreeing for length at least
\begin{equation}
 d_{X_{\Gamma'}}(\widetilde f(x),\widetilde f(v))-C.
 \label{eq:common-prefix-bcc}
\end{equation}
Here and below the boundary maps induced by the markings are understood, so the tightened image of a ray has the same endpoint in $\partial F_N$.

\begin{lem}[Finite locality under change of marking]\label{lem:marking-locality}
Let $\Gamma,\Gamma'$, $\widetilde f$, $\lambda,\kappa,C$ be as above.  There are constants $a_0\ge1$ and $b_0\ge0$, depending only on $f$, with the following property.  Let $x\in V X_\Gamma$, let $D\ge0$, and let $R\in\N$ satisfy $R\ge a_0D+b_0$.  Suppose that $S,S'\in\calH(\partial F_N)$ satisfy
\[
 x\in T_\Gamma(S)\cap T_\Gamma(S')
\]
and
\begin{equation}
 T_\Gamma(S)\cap\operatorname{Ball}_{X_\Gamma}(x,R)
 =T_\Gamma(S')\cap\operatorname{Ball}_{X_\Gamma}(x,R).
 \label{eq:same-source-ball}
\end{equation}
Then $T_{\Gamma'}(S)$ and $T_{\Gamma'}(S')$ contain exactly the same edges $e'$ satisfying
\begin{equation}
 d_{X_{\Gamma'}}(\widetilde f(x),e')\le D.
 \label{eq:near-target-edge}
\end{equation}
\end{lem}

\begin{proof}
For every oriented edge $e$ pointing away from $x$ whose terminal vertex $v$ satisfies $d_{X_\Gamma}(x,v)=R$, let $U_e$ be the component of $X_\Gamma\setminus e^\circ$ lying beyond $e$, and let
\[
 \operatorname{Sh}_x(e):=\partial U_e\subseteq\partial F_N.
\]
These depth-$R$ shadows form a finite clopen partition of $\partial F_N$.  Since $x\in T_\Gamma(S)$, one has
\[
 S\cap\operatorname{Sh}_x(e)\ne\varnothing
 \quad\Longleftrightarrow\quad
 v\in T_\Gamma(S).
\]
The forward implication follows from convexity.  Conversely, since $T_\Gamma(S)=\Hull_{X_\Gamma}(S)$ has no degree-one vertices and has ideal boundary $S$, the segment from $x$ to $v$ extends in $T_\Gamma(S)\cap U_e$ to a ray whose endpoint lies in $S$.  Hence~\eqref{eq:same-source-ball} says that $S$ and $S'$ meet the same depth-$R$ shadows.

Now let $\xi,\eta$ lie in the same depth-$R$ shadow.  Their geodesic rays from $x$ agree through the corresponding terminal vertex $v$.  By~\eqref{eq:common-prefix-bcc} and~\eqref{eq:qi-lower}, the rays in $X_{\Gamma'}$ from $\widetilde f(x)$ to $\xi$ and $\eta$ agree for length at least
\begin{equation}
 \lambda^{-1}R-\kappa-C.
 \label{eq:target-common-prefix}
\end{equation}
Choose, for example,
\[
 a_0=2\lambda,
 \qquad
 b_0=\lambda(\kappa+C+2).
\]
Then $R\ge a_0D+b_0$ implies that the quantity in~\eqref{eq:target-common-prefix} is greater than $D+1$.  Hence no edge satisfying~\eqref{eq:near-target-edge} separates two points of one depth-$R$ shadow; each of its boundary halfspaces is therefore a union of such shadows.

An edge $e'$ lies in $T_{\Gamma'}(S)$ exactly when $S$ meets both boundary halfspaces it determines.  Since $S$ and $S'$ meet the same depth-$R$ shadows, they contain the same edges satisfying~\eqref{eq:near-target-edge}.
\end{proof}

We also compare the two hull realizations uniformly.

\begin{lem}[Uniform comparison of hulls]\label{lem:uniform-hull-comparison}
Let $\Gamma,\Gamma'$ and $\widetilde f:X_\Gamma\to X_{\Gamma'}$ be as above.  There exists $H\ge0$, depending only on $f$ and a chosen homotopy inverse, such that for every $S\in\calH(\partial F_N)$,
\[
 d_{\rm Haus}\bigl(\widetilde f(T_\Gamma(S)),T_{\Gamma'}(S)\bigr)\le H.
\]
In particular, for every vertex $y'\in T_{\Gamma'}(S)$ there is a vertex $x\in T_\Gamma(S)$ with
\begin{equation}
 d_{X_{\Gamma'}}(\widetilde f(x),y')\le H+1.
 \label{eq:center-close}
\end{equation}
\end{lem}

\begin{proof}
Every point of $T_\Gamma(S)$ lies on a bi-infinite geodesic in $X_\Gamma$ with endpoints in $S$.  Since $\widetilde f$ is a quasi-isometry between trees, the image of each such line stays uniformly close to the line in $X_{\Gamma'}$ with the same ideal endpoints (equivalently, the same conclusion follows from bounded cancellation; see~\cite[Lemma~2.3.1]{BFH00}).  Hence $\widetilde f(T_\Gamma(S))$ lies in a uniform neighborhood of $T_{\Gamma'}(S)$.

For the reverse inclusion choose a marking-preserving homotopy inverse $g$ and equivariant lift $\widetilde g$.  The same argument for $g$, together with bounded distance of $\widetilde f\widetilde g$ from the identity and Lipschitzness of $\widetilde f$, gives the reverse uniform bound.  Enlarge $H$ also by the maximum length of an image edge, so that replacing a source point by a source vertex yields~\eqref{eq:center-close}.
\end{proof}

\subsection{Marking independence}

\begin{thm}[Change of marking for round complexity]\label{thm:marking-invariance}
Let $\Gamma$ and $\Gamma'$ be finite marked core graphs for $F_N$.  There exist constants $a,a'\ge1$, $b,b'\ge0$, and $M,M'\ge1$, depending only on the two marked graphs and on fixed marking-preserving homotopy equivalences between them, such that for every nonempty subset lamination $L\subseteq\calH(\partial F_N)$ and every $n\ge1$,
\begin{equation}
 R_L^{\Gamma'}(n)
 \le M R_L^\Gamma(\lceil an+b\rceil)
 \label{eq:marking-upper}
\end{equation}
and
\begin{equation}
 R_L^\Gamma(n)
 \le M' R_L^{\Gamma'}(\lceil a'n+b'\rceil).
 \label{eq:marking-lower}
\end{equation}
Consequently
\[
 \boxed{R_L^\Gamma\simeq_{\rm lin}R_L^{\Gamma'}.}
\]
\end{thm}

\begin{proof}
Fix $f$, $\widetilde f$, and $H$ as in Lemma~\ref{lem:uniform-hull-comparison}.  Put $H_0=H+1$.  Since the $F_N$-actions on $X_\Gamma$ and $X_{\Gamma'}$ are cocompact and the trees are locally finite, the set
\[
 \mathcal D=
 F_N\backslash
 \{(x,y')\in V X_\Gamma\times V X_{\Gamma'}:
 d_{X_{\Gamma'}}(\widetilde f(x),y')\le H_0\}
\]
is finite.  Put $M=|\mathcal D|$.

Fix $n\ge1$ and choose
\begin{equation}
 R=\left\lceil a_0(n+H_0+1)+b_0\right\rceil,
 \label{eq:source-radius-choice}
\end{equation}
where $a_0,b_0$ are supplied by Lemma~\ref{lem:marking-locality}.  For each translation class in $\mathcal R_n^{\Gamma'}(L)$, choose an occurrence
\[
 \bigl(T_{\Gamma'}(S)\cap\operatorname{Ball}_{X_{\Gamma'}}(y',n),y'\bigr)
\]
with $S\in L$ and $y'\in V T_{\Gamma'}(S)$, and then choose, using~\eqref{eq:center-close}, a vertex $x\in V T_\Gamma(S)$ such that
\[
 d_{X_{\Gamma'}}(\widetilde f(x),y')\le H_0.
\]
Associate to the chosen target class the pair consisting of the relative-center type $[(x,y')]\in\mathcal D$ and the source pattern
\[
 \bigl(T_\Gamma(S)\cap\operatorname{Ball}_{X_\Gamma}(x,R),x\bigr)
 \in\mathcal R_R^\Gamma(L).
\]

This association is injective.  Suppose two chosen target classes have the same relative-center type and the same rooted source-pattern class.  Translate one occurrence so that the two center pairs agree.  Because the deck action on $X_\Gamma$ is free, the same translation also aligns the two rooted source patterns.  Lemma~\ref{lem:marking-locality}, with
\[
 D=n+H_0+1,
\]
then makes the target hulls agree on the radius-$n$ ball about $y'$, so the pointed target patterns coincide.

Therefore
\[
 R_L^{\Gamma'}(n)\le |\mathcal D|R_L^\Gamma(R).
\]
Since~\eqref{eq:source-radius-choice} has the form $R\le\lceil an+b\rceil$ for constants depending only on the markings, monotonicity of $R_L^\Gamma$ gives~\eqref{eq:marking-upper}.  Applying the same argument to a marking-preserving homotopy equivalence $\Gamma'\to\Gamma$ gives~\eqref{eq:marking-lower}.
\end{proof}

\begin{cor}[Basis independence]\label{cor:basis-round-invariance}
Let $A$ and $B$ be free bases of $F_N$, and let $R_L^A$ and $R_L^B$ be the corresponding round-complexity functions for a nonempty subset lamination $L$.  Then
\[
 R_L^A\simeq_{\rm lin}R_L^B.
\]
In particular, each of the following properties is independent of the chosen free basis:
\begin{enumerate}[(1)]
\item $R_L(n)=O(n^k)$ for any fixed $k\ge0$, and hence the polynomial growth degree when it is finite;
\item $R_L(n)=O(n)$;
\item subexponential growth, equivalently
\[
 \limsup_{n\to\infty}\frac1n\log R_L(n)=0;
\]
\item the property $R_L(n)/n\to\infty$.
\end{enumerate}
Consequently the linear-complexity hypothesis in Theorem~\ref{mainthm:complexity} is basis-independent.
\end{cor}

\begin{proof}
Apply Theorem~\ref{thm:marking-invariance} to the two marked roses corresponding to $A$ and $B$.  The first three consequences follow from $\simeq_{\rm lin}$ and monotonicity.  For~(4), suppose, for example, that $R_L^B(m_j)/m_j$ stays bounded along $m_j\to\infty$.  Use an inequality $R_L^A(n)\le M R_L^B(\lceil an+b\rceil)$ and choose integers $n_j\asymp m_j/a$ with $\lceil an_j+b\rceil\le m_j$.  Monotonicity then gives $R_L^A(n_j)/n_j=O(1)$.  Thus superlinearity for $A$ implies superlinearity for $B$, and the reverse implication is symmetric.
\end{proof}

\begin{rem}[Outer-space representatives]\label{rem:outer-space-round}
Unit edge lengths keep radii integral.  If a fixed finite marked core graph is given arbitrary positive edge lengths, its metric and combinatorial distances on vertices are bi-Lipschitz equivalent; hence the analogous complexity defined using vertex-centered metric balls has the same linear-rescaling growth type.  Together with Theorem~\ref{thm:marking-invariance}, this comparison shows that the coarse growth type is independent of the chosen marked metric graph, and hence of the point of Culler--Vogtmann Outer space used to realize $F_N$~\cite{CV86}.
\end{rem}

\begin{rem}[What is not asserted]\label{rem:marking-noninvariants}
The theorem does not preserve a numerical exponential base such as
\[
 \limsup_{n\to\infty}R_L^\Gamma(n)^{1/n},
\]
because a change of marking may rescale radius multiplicatively.  Nor does the argument compare the one-step excess $E_L(n)$ or the finite-level nullity $q_L(n)$, which retain more marking-sensitive finite-scale information.
\end{rem}

\section{Disclosure of AI use}

Preparation of this paper substantially relied on chats with ChatGPT, but the author verified all the proofs given, reviewed and edited the content as needed and takes full responsibility for the content of the paper.


\begin{thebibliography}{99}

\bibitem[BFH00]{BFH00}
M. Bestvina, M. Feighn and M. Handel,
\emph{The Tits alternative for $\operatorname{Out}(F_n)$. I. Dynamics of exponentially-growing automorphisms},
Ann. of Math. (2) \textbf{151} (2000), no.~2, 517--623.
\bibitem[Bon86]{Bon86}
F. Bonahon,
\emph{Bouts des vari\'et\'es hyperboliques de dimension 3},
Ann. of Math. (2) \textbf{124} (1986), no.~1, 71--158.
\bibitem[Bon88]{Bon88}
F. Bonahon,
\emph{The geometry of Teichm\"uller space via geodesic currents},
Invent. Math. \textbf{92} (1988), no.~1, 139--162.
\bibitem[Bon91]{Bon91}
F. Bonahon,
\emph{Geodesic currents on negatively curved groups},
in \emph{Arboreal group theory} (Berkeley, CA, 1988),
Math. Sci. Res. Inst. Publ., vol.~19,
Springer, New York, 1991, 143--168.
\bibitem[Bos84]{Bos84}
M. Boshernitzan,
\emph{A unique ergodicity of minimal symbolic flows with linear block growth},
J. Analyse Math. \textbf{44} (1984), 77--96.
\bibitem[BH99]{BH99}
M. Bridson and A. Haefliger,
\emph{Metric Spaces of Non-Positive Curvature},
Grundlehren der mathematischen Wissenschaften, vol.~319,
Springer-Verlag, Berlin, 1999.
\bibitem[CCG24]{CCG24}
P. Cecchi Bernales, M. I. Cortez and J. G\'omez,
\emph{Invariant measures of Toeplitz subshifts on non-amenable groups},
Ergodic Theory Dynam. Systems \textbf{44} (2024), no.~11, 3186--3215.
\bibitem[CD26]{CD26}
P. Cecchi Bernales and S. Donoso,
\emph{Strong orbit equivalence and superlinear complexity},
Israel J. Math., to appear; arXiv:2009.13654.
\bibitem[Coo87]{Coo87}
D. Cooper,
\emph{Automorphisms of free groups have finitely generated fixed point sets},
J. Algebra \textbf{111} (1987), no.~2, 453--456.
\bibitem[CHL08]{CHL08}
T. Coulbois, A. Hilion and M. Lustig,
\emph{$\mathbb R$-trees and laminations for free groups I: algebraic laminations},
J. Lond. Math. Soc. (2) \textbf{78} (2008), no.~3, 723--736.
\bibitem[CV86]{CV86}
M. Culler and K. Vogtmann,
\emph{Moduli of graphs and automorphisms of free groups},
Invent. Math. \textbf{84} (1986), no.~1, 91--119.
\bibitem[CK19]{CK19}
V. Cyr and B. Kra,
\emph{Counting generic measures for a subshift of linear growth},
J. Eur. Math. Soc. \textbf{21} (2019), no.~2, 355--380.
\bibitem[Dow91]{Dow91}
T. Downarowicz,
\emph{The Choquet simplex of invariant measures for minimal flows},
Israel J. Math. \textbf{74} (1991), no.~2--3, 241--256.
\bibitem[Gri73]{Gri73}
C. Grillenberger,
\emph{Constructions of strictly ergodic systems. I. Given entropy},
Z. Wahrscheinlichkeitstheorie und Verw. Gebiete \textbf{25} (1972/73), 323--334.
\bibitem[Kap99]{Kap99}
I. Kapovich,
\emph{A non-quasiconvexity embedding theorem for hyperbolic groups},
Math. Proc. Cambridge Philos. Soc. \textbf{127} (1999), no.~3, 461--486.
\bibitem[Kap06]{Kap06}
I. Kapovich,
\emph{Currents on free groups},
in \emph{Topological and asymptotic aspects of group theory},
Contemp. Math., vol.~394, Amer. Math. Soc., Providence, RI, 2006, 149--176.
\bibitem[KapLus09]{KapLus09}
I. Kapovich and M. Lustig,
\emph{Geometric intersection number and analogues of the curve complex for free groups},
Geom. Topol. \textbf{13} (2009), no.~3, 1805--1833.
\bibitem[KapLus10]{KapLus10}
I. Kapovich and M. Lustig,
\emph{Intersection form, laminations and currents on free groups},
Geom. Funct. Anal. \textbf{19} (2010), no.~5, 1426--1467.
\bibitem[Kap17]{Kap17}
I. Kapovich,
\emph{An integral weight realization theorem for subset currents on free groups},
Topology Proc. \textbf{50} (2017), 213--236.
\bibitem[Kap26]{Kap26}
I. Kapovich,
\emph{On the Hausdorff dimension and attracting laminations for fully irreducible automorphisms of free groups},
J. Groups Complex. Cryptol. \textbf{18} (2026), no.~1, 5:1--5:32.
\bibitem[KN13]{KN13}
I. Kapovich and T. Nagnibeda,
\emph{Subset currents on free groups},
Geom. Dedicata \textbf{166} (2013), 307--348.
\bibitem[KM04]{KM04}
A. S. Kechris and B. D. Miller,
\emph{Topics in Orbit Equivalence},
Lecture Notes in Mathematics, vol.~1852,
Springer, Berlin, 2004.
\bibitem[KL24]{KL24}
D. Kielak and M. Linton,
\emph{Virtually free-by-cyclic groups},
Geom. Funct. Anal. \textbf{34} (2024), 1580--1608.
\bibitem[Martin95]{Martin95}
R. Martin,
\emph{Non-uniquely ergodic foliations of thin type, measured currents and automorphisms of free groups},
Ph.D. thesis, University of California, Los Angeles, 1995.
\bibitem[Phe01]{Phe01}
R. R. Phelps,
\emph{Lectures on Choquet's Theorem},
2nd ed., Lecture Notes in Mathematics, vol.~1757,
Springer, Berlin, 2001.
\bibitem[Que10]{Que10}
M. Queff\'elec,
\emph{Substitution Dynamical Systems---Spectral Analysis},
2nd ed., Lecture Notes in Mathematics, vol.~1294,
Springer, Berlin, 2010.
\bibitem[Sas22]{Sas22}
D. Sasaki,
\emph{Subset currents on surfaces},
Mem. Amer. Math. Soc. \textbf{278} (2022), no.~1368.
\bibitem[Yan12]{Yan12}
W.-Y. Yang,
\emph{Limit sets of relatively hyperbolic groups},
Geom. Dedicata \textbf{156} (2012), 1--12.

\end{thebibliography}
\end{document}